\documentclass[journal]{IEEEtran}
\ifCLASSINFOpdf
\else
\fi
\usepackage[dvipsnames]{xcolor}
\usepackage{url}
\usepackage{graphicx}
\usepackage{mathtools}
\usepackage{amssymb}
\usepackage{amsfonts}
\usepackage{amsmath}
\usepackage{tikz}
\usepackage{multirow}
\usepackage{lipsum}
\usepackage{mathtools}
\usepackage{multirow}
\usepackage{pgfplots}
\usepgfplotslibrary{fillbetween}
\usetikzlibrary{calc}
\usetikzlibrary{angles, quotes}
\usepackage{enumitem}
\usepackage{subfigure}
\usepackage{rotate}
\usepackage{fontawesome}
\usepackage{amsthm}
\usepackage{todonotes}
\usepackage{makecell}
\usepackage{bbm}
\usepackage{pgfplotstable,filecontents}
\usepackage{float}
\usepackage[noadjust]{cite}
\usepackage{array}
\usepackage{balance}
\usepackage[linesnumbered,ruled]{algorithm2e}

\floatstyle{ruled}
\newfloat{model}{thp}{lop}
\floatname{model}{Model}

\newtheorem{definition}{Definition}
\newtheorem{theorem}{Theorem}

\newtheorem{lemma}{Lemma}

\newtheorem{remark}{Remark}

\newcommand{\minimize}[1]{\underset{{#1}}{\text{min}}}

\newcommand{\braceit}[1]{\left({#1}\right)}
\usepackage{euscript}
\newcommand{\set}{\EuScript}
\usepackage{mathtools}
\newcommand{\norm}[1]{\lVert#1\rVert}
\newcommand{\normtwo}[2]{#1\lVert#2#1\rVert}

\newcommand\mydots{\hbox to 0.75em{.\hss.\hss.}}

\tikzset{myarr/.style={{Circle[black,length=4pt]}-{Circle[black,length=4pt]},shorten <=-2.5pt,shorten >=-2.5pt}}

\DeclareMathSymbol{\shortminus}{\mathbin}{AMSa}{"39}
\newcommand{\mysum}[1]{\underset{{{#1}}}{\textstyle\sum}}

\makeatletter
\newcommand{\oset}[3][0ex]{%
  \mathrel{\mathop{#3}\limits^{
    \vbox to#1{\kern-2\ex@
    \hbox{$\scriptstyle#2$}\vss}}}}
\makeatother
\newcommand{\optimal}[1]{\oset{\scalebox{.6}{$\star$}}{#1}}

\newcommand{\newtext}[1]{\textcolor{black}{#1}}

\newcommand{\M}{{\mathcal{M}}}
\newcommand{\Mdp}{{\tilde{\mathcal{M}}}}

\newcommand{\RR}{\mathbb{R}}

\newcommand{\cD}{\mathcal{D}}

\newcommand{\cR}{\mathcal{R}}

\def\opt{\overset{\star}}

\pgfmathdeclarefunction{gauss}{2}{%
\pgfmathparse{1/(#2*sqrt(2*pi))*exp(-((x-#1)^2)/(2*#2^2))}%
}
\pgfmathdeclarefunction{gauss_reverse}{2}{%
\pgfmathparse{1/(#2*sqrt(2*pi))*exp(-((x-#1)^2)/(2*#2^2))}%
}

\definecolor{maincolor}{HTML}{032F99}
\definecolor{blue}{RGB}{31,64,122}
\definecolor{red}{HTML}{e05a87} 

\usepackage[colorlinks=true, citecolor=red, linkcolor=blue]{hyperref}

\begin{document}

\title{Differentially Private Optimal Power Flow \\ for Distribution Grids}
\author{Vladimir~Dvorkin~Jr.,~\IEEEmembership{Student member,~IEEE,}
        Ferdinando~Fioretto,
        Pascal~Van~Hentenryck,~\IEEEmembership{Member,~IEEE,}
        Pierre~Pinson,~\IEEEmembership{Fellow,~IEEE,}
        and~Jalal~Kazempour,~\IEEEmembership{Senior Member,~IEEE}
        
\thanks{V. Dvorkin Jr., P. Pinson,  and J. Kazempour are with the Technical University of Denmark,  Kgs. Lyngby, Denmark. F. Fioretto is with the Syracuse University, Syracuse, NY, USA. P. Van Hentenryck is with the Georgia Institute of Technology, Atlanta, GA, USA.}
}
\maketitle

\begin{abstract}
Although distribution grid customers are obliged to share their consumption data with distribution system operators (DSOs), a possible leakage of this data is often disregarded in operational routines of DSOs. This paper introduces a privacy-preserving optimal power flow (OPF) mechanism for distribution grids that secures customer privacy from unauthorised access to OPF solutions, e.g., current and voltage measurements. The mechanism is based on the framework of \emph{differential privacy} that allows to control the participation risks of individuals in a dataset by applying a carefully calibrated noise to the output of a computation. Unlike existing private mechanisms, this mechanism does not apply the noise to the optimization parameters or its result. Instead, it optimizes OPF variables as affine functions of the random noise, which weakens the correlation between the grid loads and OPF variables. To ensure feasibility of the randomized OPF solution, the mechanism makes use of chance constraints enforced on the grid limits. The mechanism is further extended to control the optimality loss induced by the random noise, as well as the variance of OPF variables. The paper shows that the differentially private OPF solution does not leak customer loads up to specified parameters. 
\end{abstract}
\IEEEpeerreviewmaketitle
\begingroup
\allowdisplaybreaks

\section{Introduction}
\IEEEPARstart{T}{he} increasing observability of distribution grids enables advanced operational practices for distribution system operators (DSOs). In particular, high-resolution voltage and current measurements available to DSOs allow for continuously steering the system operation towards an optimal power flow (OPF) solution \cite{dall2016optimal,bolognani2017fast,mieth2018data}. However, when collected, these measurements expose distribution grid customers to privacy breaches. Several studies have shown that the measurements of OPF variables can be used by an adversary to identify the type of appliances and load patterns of grid customers \cite{duarte2012non,cole1998data}. The public response to these privacy risks has been demonstrated by the Dutch Parliament's decision to thwart the deployment of smart meters until the privacy concerns are resolved \cite{erkin2013privacy}.

Although grid customers tend to entrust DSOs with their data in exchange for a reliable supply, their privacy rights are often disregarded in operational routines of DSOs. To resolve this issue, this paper augments the OPF computations with the preservation of customer privacy in the following sense. 

\begin{definition}[Customer privacy]\label{def_customer_privacy}
The right of grid customers to be secured from an unauthorized disclosure of sensitive information that can be inferred from the OPF solution. 
\end{definition}

To ensure this right, privacy needs to be rigorously quantified and guaranteed. Differential privacy (DP) \cite{dwork2006calibrating} is a strong privacy notion that quantifies and bounds privacy risks in computations involving sensitive datasets. By augmenting the computations with a carefully calibrated {\it random} noise, 
a DP {\it mechanism} guarantees that the noisy results do not disclose the attributes of individual items in a dataset. Chaudhuri {\it et al.} \cite{chaudhuri2011differentially} and Hsu {\it et al.} \cite{hsu2014privately} introduced several mechanisms to solve optimization models while preventing the recovery of the input data from optimization results. These mechanisms apply noise to either the parameters or the results of an optimization. The applied noise, however, fundamentally alters the optimization problem of interest. Therefore,  the direct application of these mechanisms to OPF problems has been limited. First, they may fail to provide a {\it feasible} solution for constrained optimization problems. To restore feasibility, they require an additional level of complexity such as the post-processing steps proposed in \cite{fioretto2019differential,mak2019privacy}. Second, although these mechanisms provide bounds on the worst-case performance, they do not consider the {\it optimality loss} as a control variable. As a result, they cannot provide appropriate trade-offs between the expected and the worst-case mechanism performances. Finally, the previously proposed mechanisms overlook the impact of the noise on the {\it variance} of the optimization results. Hence, their direct application to OPF problems may lead to undesired overloads of system components \cite{baghsorkhi2012impact}.

{\it Contributions:} To overcome these limitations, this paper proposes a novel differentially private OPF mechanism that does not add the noise to the optimization parameters or to the results. Instead, it obtains DP by optimizing OPF variables as {\it affine functions} of the noise, bypassing the above-mentioned theoretical drawbacks. More precisely, the paper makes the following contributions: 
\begin{enumerate}[leftmargin=0.45cm, label=\arabic{enumi}.]
    \item The proposed mechanism produces a random OPF solution that follows a Normal distribution and guarantees $(\varepsilon,\delta)-$differential privacy \cite{dwork2006calibrating}. Parameters $\varepsilon$ and $\delta$, respectively, bound the multiplicative and additive differences between the probability distributions of OPF solutions obtained on adjacent datasets (i.e., differing in at most one load value). 
    The mechanism is particularly suitable for protecting grid loads from unauthorized access to OPF solutions, as fine-tuned $(\varepsilon,\delta)$ values make randomized OPF solutions similar, irrespective of the used load dataset. 
    \item The mechanism enforces chance constraints on random OPF variables to guarantee solution feasibility for a given constraint satisfaction probability. This way, it does not require a post-processing step to restore OPF feasibility, as in \cite{fioretto2019differential} and \cite{mak2019privacy}. Since the OPF variables are affine in the Gaussian noise, the chance constraints are reformulated into computationally efficient second-order cone constraints.
    \item The mechanism enables the control of random OPF outcomes without weakening the DP guarantees. Using results from stochastic programming \cite{shapiro2009lectures}, the optimality loss induced by the noise is controlled using Conditional Value-at-Risk (CVaR) risk measure, enabling a trade-off between the expected and the worst-case performance. Furthermore, with a variance-aware control from \cite{bienstock2019variance}, the mechanism attains DP with a smaller variance of OPF variables. 
\end{enumerate}

\newtext{
{\it Broader Impact:} 
Distribution OPF proposals have been around for at least a decade, though their adoption in real operations is complicated by the need of utilizing load datasets, which raises significant privacy concerns by many regulators worldwide. The adoption of the proposed mechanism, in turn, extends standard OPF models to enable a privacy-cognizant utilization\footnote{\newtext{Note that privacy concerns a safe utilization of the data, not its storage, which falls within the field of cyber-security.}} of this data, thus facilitating the digitization of the energy sector. The mechanism treats the DSOs as trust-worthy parties and places them on the same ground with the digital service providers, e.g. Amazon, enabling the regulation and securing legal responsibility of the digitalized distribution grids under modern data protection and privacy standards, including the General Data Protection Rights (GDPR) in the European Union, the California Consumer Privacy Act (CCPA) and the New York Privacy Act (NYPA) in the United States. Moreover, the mechanism provides the means to hedge the financial risks of the DSOs by avoiding the cost incurred by privacy violations, such as legal costs, as it relies on a strong quantification of privacy and co-optimizes the joint cost of energy supply and privacy.
}

{\it Related Work:} Thanks to its strong privacy guarantees, DP has been recently applied to private OPF computations. In particular, the mechanism of Zhou {\it et al.} \cite{zhou2019differential} releases aggregated OPF statistics, e.g., aggregated load, while ensuring the privacy for individual loads, even if all but one loads are compromised. The proposals by Fioretto {\it et al.} \cite{fioretto2019differential} and Mak {\it et al.} \cite{mak2019privacy} provide a differentially private mechanism to release high-fidelity OPF datasets (e.g., load and network parameters) from real power systems while minimizing the risks of disclosing the actual system parameters. The mechanisms, however, are meant for the private release of aggregate statistics and input datasets and do not provide the OPF solution itself. 

Private OPF computations have also been studied in a decentralized and distributed setting. Dvorkin {\it et al.} \cite{dvorkin2019differentially} designed a distributed OPF algorithm with a differentially private exchange of coordination signals, hence preventing the leakage of the sensitive information contained in the algorithm subproblems. Han {\it et al.} \cite{han2016differentially} proposed a privacy-aware distributed coordination scheme for electrical vehicle charging. The privacy frameworks in \cite{dvorkin2019differentially} and \cite{han2016differentially}, however, are not suitable for centralized computations \newtext{and solely focus on the privacy leakage through the exchange of coordination signals. Moreover, to negate the privacy loss induced at every iteration, they require scaling the parameters of the random perturbation, thus involving larger optimality losses and poorer convergence. The centralized mechanisms proposed in this work, however, allow obtaining the private OPF solution in a single computation run.} In distribution systems, Zhang {\it et al.} \cite{zhang2017cost}, among other proposals reviewed in \cite{erkin2013privacy}, designed a privacy-aware optimization of behind-the-meter energy storage systems to prevent the leakage of consumption data from the smart meter readings. However, they disregard OPF feasibility of distribution systems, which has to be preserved in all circumstances. 

{\it Paper Organization:} Following the preliminaries in Section \ref{sec:preliminaries}, Section \ref{sec:goals} formalizes the privacy goals and provides an overview of the proposed solution. Section \ref{DP_OPF_mechansim} provides the formulation of the proposed privacy-preserving OPF mechanism and its properties, whereas Section \ref{sec:properties} presents its extensions. Section  \ref{sec:experiments} provides numerical experiments and Section \ref{sec_conc} concludes. The proofs are relegated to the appendix.

\section{Preliminaries}\label{sec:preliminaries}
\subsection{Optimal Power Flow Problem}

The paper considers a low-voltage radial distribution grid with
controllable distributed energy resources (DERs). A DSO is responsible for
controlling the DERs and supplying power from the high-voltage grid while meeting the technical limits of the grid. The grid is 
modeled as a graph  $\Gamma\braceit{\set{N},\set{L}}$, where
$\set{N}=\{0, 1, \ldots, n\}$ is the set of nodes and $\set{L}=\set{N}\setminus \{0\}$ is the set of lines connecting those nodes.
The root node, indexed by $0$, is a substation with a large capacity
and fixed voltage magnitude $v_0 = 1$. The radial topology, depicted in 
Fig.~\ref{grid_topology}, associates each node $i \in \set{N}$ with the sets $\set{U}_{i}$ and $\set{D}_{i}$ of, respectively, upstream and downstream nodes, as well as with the set $\set{R}_{i}$ of nodes on the path to the root node. 


Each node $i$ is characterized by its fixed active $d_{i}^{p}$ and reactive $d_{i}^{q}$ power load and by its voltage magnitude $v_{i}\in[\underline{v}_{i},\overline{v}_{i}]$. 
For modeling convenience, the voltage variables are substituted by $u_{i}=v_{i}^{2},\;\forall i\in\set{N}$.
A controllable DER sited at node $i$ outputs an amount of active 
$g_{i}^{p}\in[\underline{g}_{i}^{p},\overline{g}_{i}^{p}]$ and
reactive $g_{i}^{q}\in[\underline{g}_{i}^{q},\overline{g}_{i}^{q}]$
power. Its costs are linear with a cost coefficient $c_{i}$. To model the relation between the active and reactive DER power output, the constant power factor $\tan{\phi_{i}}$ is assumed for each node $i$. 
The active and reactive power flows, $f_{\ell}^{p}$ and $f_{\ell}^{q}$, $\forall \ell \in \set{L}$, respectively, are constrained by the apparent power limit $\overline{f}_{\ell}$, and each line is characterized by its resistance $r_{\ell}$ and reactance $x_{\ell}$. The deterministic OPF model is formulated as:
\label{model:det_opf}
\begin{subequations}\label{det_OPF}
\begin{align}
    \text{D-OPF}\colon\;&\minimize{g^{\dag},f^{\dag},u}\;\mysum{i\in\set{N}}c_{i}g_{i}^{p}\label{det_OPF_obj}\\
    \text{s.t.}
    \;\;&g_{0}^{\dag}=\mysum{i\in\set{D}_{0}}(d_{i}^{\dag}-g_{i}^{\dag}),\;u_{0}=1,\label{det_OPF_root}\\
    &f_{\ell}^{\dag}=d_{\ell}^{\dag}-g_{\ell}^{\dag} + \mysum{i\in\set{D}_{\ell}}(d_{i}^{\dag}-g_{i}^{\dag}), \;\forall \ell\in\set{L},\label{det_OPF_flow}\\
    &u_{i} = u_{0} - 2\mysum{\ell \in\set{R}_{i}}(f_{\ell}^{p}r_{\ell}+f_{\ell}^{q}x_{\ell}),\;\forall i\in\set{L},\label{det_OPF_voltage}\\
    &(f_{\ell}^{p})^2 + (f_{\ell}^{q})^2 \leqslant\overline{f}_{\ell}^2,\;\forall \ell\in\set{L},\label{det_OPF_flow_limit}\\
    &\underline{g}_{i}^{\dag} \leqslant g_{i}^{\dag}\leqslant \overline{g}_{i}^{\dag},\;\forall i\in\set{N},\label{det_OPF_gen_limit}\\
    &\underline{v}_{i}^2 \leqslant u_{i} \leqslant \overline{v}_{i}^2,\;\forall i\in\set{N}\setminus \{0\},\label{det_OPF_voltage_limit}
\end{align}
\end{subequations}
where superscript $\dag=\{p,q\}$ indexes active and reactive power. \newtext{The objective is to minimize the total operational cost subject to the OPF equations \eqref{det_OPF_root}--\eqref{det_OPF_voltage}, that balance the grid based on the \emph{LinDistFlow} AC power flow equations \cite[equations (9)]{baran1989optimal} for distribution grids, and grid limits \eqref{det_OPF_flow_limit}--\eqref{det_OPF_voltage_limit}. Equation \eqref{det_OPF_root} requires the total mismatch between power generation and loads in the distribution grid to be compensated for by the power from the substation at the root node. Equation \eqref{det_OPF_flow} requires the balance between the power flow along every edge $\ell$, power mismatch at the in-flow node $\ell$ as well as that at the downstream nodes. The last term in \eqref{det_OPF_flow} can be also rewritten as the sum of power flows in the adjacent downstream lines, but kept as it is in the interest of the subsequent derivations. Last, equation \eqref{det_OPF_voltage} models the voltage drop along the path from the root node to the node of interest.}  

\newtext{Although OPF equations (1b)-(1d) establish the affine relation between the OPF variables, which is necessary for the subsequent chance-constrained formulation, they neglect distribution grid losses. The losses, however, can be included in an affine manner using various linearization techniques, such as in \cite{roald2016optimization}, \cite{bernstein2018load} and \cite{mieth2019distribution} to mention but a few examples.}

\begin{figure}
\centering
\resizebox{0.45\textwidth}{!}{%
\includegraphics{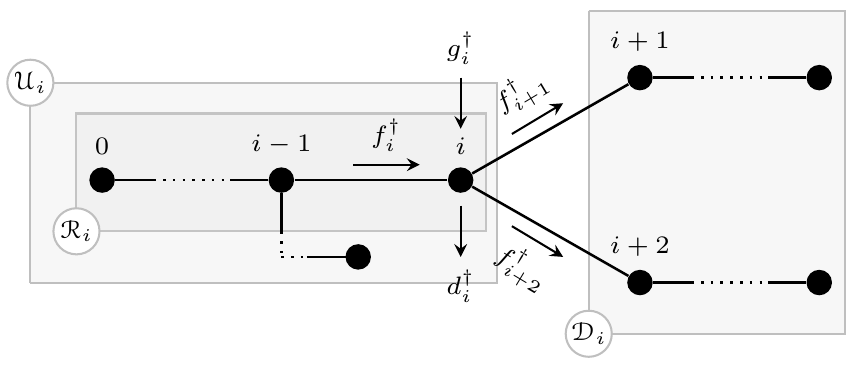}
}
\caption{Topology of the distribution grid relative to node $i$, $\dag=\{p,q\}$.}
\label{grid_topology}
\end{figure}

\subsection{Differential Privacy}
\label{sec:fifferential_privacy}

The paper uses the framework of \emph{differential privacy} \cite{dwork2006calibrating} to quantify and control the privacy risks of the customer loads.  It considers datasets $D \in \mathbb{R}^n$ as $n$-dimensional vectors describing the \emph{active} load values, denoted by $d_i$ for each node $i$. To protect the participation of the load in the $i^{\text{th}}$ entry of the dataset, the following \emph{adjacency relation} is introduced:
\begin{equation*} 
\label{eq:adj_rel} 
    D \sim_\beta D' \Leftrightarrow \exists i
    \textrm{~s.t.~} | d_i - {d}_{i}' | \leqslant \beta_i \;\newtext{\text{and}}\;
    d_j = {d}_{j}', \forall j \neq i,
\end{equation*} 
where $D$ and $D'$ are two adjacent datasets, $\beta \in \RR^n$ is a vector of positive real values, and values $d_i$ and $d_i'$ are the load values corresponding to customer $i$ in $D$ and $D'$, respectively.  The adjacency relates two load vectors that differ in at most one item, at position $i$, by a value not greater than $\beta_i$. 

If a mechanism satisfies the definition of differential privacy, it returns similar results on adjacent datasets in a probabilistic sense. This intuition is formalized in the following definition.

\begin{definition}[Differential Privacy]
  \label{eq:dp_def}
Given a value $\beta \in \RR^n_{+}$, a randomized mechanism $\Mdp \!:\! \cD \!\to\! \cR$ with domain $\cD$ and range $\cR$ is $(\varepsilon, \delta)$-differential private if, for any output $s \subseteq \cR$ and any two adjacent inputs $D \sim_\beta D' \in \RR^n$
\begin{equation*}
  \mathbb{P}[\Mdp(D) \in s] \leqslant e^\varepsilon \mathbb{P}[\Mdp(D') \in s] + \delta,
\end{equation*}
where $\mathbb{P}$ denotes the probability over runs of $\Mdp$. 
\end{definition}

\newtext{In the context of OPF problem \eqref{det_OPF}, domain $\mathcal{D}$ includes all feasible load datasets, mechanism $\mathcal{M}$ denotes the OPF problem itself, and $\tilde{\mathcal{M}}$ is its randomized counterpart, and range $\mathcal{R}$ denotes the feasible region of the OPF problem.}

The level of privacy is controlled by DP parameters $(\varepsilon,\delta)$. The former corresponds to the maximal multiplicative difference in distributions obtained by the mechanism on adjacent datasets, whereas the latter defines the maximal additive difference. Consequently, smaller values of $\varepsilon$ and $\delta$ provide stronger privacy protection. Definition \ref{eq:dp_def} extends the metric-based differential privacy introduced by Chatzikokolakis {\it et al.} \cite{chatzikokolakis2013broadening} to control of \emph{individual} privacy risks. 

\newtext{If a mechanism satisfies Definition \ref{eq:dp_def}, it features two important properties. First, by acting on adjacent datasets $D$ and $D^{\prime}$, it provides privacy for each item $i$ {\it irrespective} of the properties of all remaining items in a dataset. Second, it is immune to the so-called side attacks, i.e., it ensures that even if an attacker acquires the data of all other users but $i$, when accessing the output $\tilde{\mathcal{M}}(D)$ of the differential private mechanism, it will not be able to infer the load value of user $i$ up to differential privacy bounds $\varepsilon$ and $\delta$ \cite{dwork2014algorithmic}}.

The differentially private design of any mechanism is obtained by means of randomization using, among others, \newtext{Laplace or Gaussian noise for numerical queries and exponential noise for the so-called non-numerical events \cite{dwork2014algorithmic}}. The DP requirements for an optimization problem are achieved by introducing a calibrated noise to the input data \cite{hsu2014privately} or to the output or objective function of the mechanism itself \cite{chaudhuri2011differentially}. 
Regardless of the strategy adopted to attain DP, the amount of noise to inject depends on the mechanism \emph{sensitivity}. 
In particular, the $L_2-$sensitivity of a deterministic mechanism $\M$ on $\beta$-adjacent datasets, denoted by $\Delta^\beta$, is defined as:
\[
    \Delta^\beta = \max_{D \sim_\beta D'} 
    \left\|  {\cal M}(D) - {\cal M}(D')\right\|_2.
\]
\newtext{This work employs the Gaussian mechanism, because the Gaussian noise allows for the exact analytic reformulation of chance constraints into tractable second-order cone constraints.} 
\begin{theorem}[Gaussian mechanism \cite{dwork2014algorithmic}]
\normalfont\label{def_gaus_mech} 
Let ${\cal M}$ be a mechanism of interest that maps datasets $D$ to $\RR^n$\newtext{, and let $\Delta^\beta$ be its $L_2-$sensitivity}. For  $\varepsilon\in(0,1)$ and $\gamma^{2}>2\ln(\frac{1.25}{\delta})$, the Gaussian mechanism that outputs ${\cal \tilde{M}}(D) ={\cal M}(D) + \xi$, with $\xi$ noise  drawn from the Normal distribution with $0$ mean and standard deviation $\sigma \geqslant \gamma\frac{\Delta^{\beta}}{\varepsilon}$ 
is $(\varepsilon,\delta)$-differentially private. 
\end{theorem}

When the DP mechanism produces solutions to an optimization problem, it is also important to quantify the {\em optimality loss}, i.e., the distance between the optimal solutions of the original mechanism $\M(D)$ and its differentially private counterpart $\Mdp(D)$ evaluated on the original dataset $D$.

\section{Problem Statement}
\label{sec:goals}

In the context of the underlying dispatch problem, the DSO collects a
dataset $D=\{d_i^{p}\}_{i \in \set{N}}$ of customer \emph{sensitive} loads and dispatches the DER according to the
solution of the OPF model \eqref{det_OPF}.  The OPF model acts as
a mechanism $\M:D\mapsto s$ that maps the dataset $D$ into an optimal OPF
solution $s^{\star}$. The solution is a tuple comprising generator set
points $\{g_i^p,g_i^q\}_{i \in\set{N}}$, power flows $\{f_\ell^p,
f_\ell^q\}_{\ell \in \set{L}}$, and voltages $\{u_i\}_{i \in \set{N}},$ as depicted on the left plane in Fig. \ref{OPF_projections}.
However, the release of $s^{\star}$ poses a privacy threat:
an adversary with access to the items in $s^{\star}$ could
\emph{decode} the customers activities
\cite{duarte2012non,cole1998data}. For instance, the voltage sags at a
node of interest discloses the activity of residential costumers
(e.g., charging an electrical vehicle). Voltages and flows (currents)
also encode information about the technology, production patterns, and
other commercial data of industrial customers \cite{fioretto2019differential}.

\begin{figure}
    \centering
    \hspace{1.5cm}\resizebox{0.42\textwidth}{!}{%
    \includegraphics{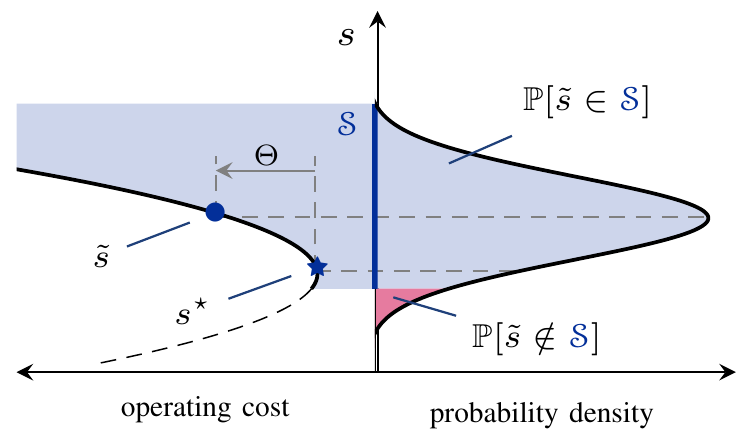}
    }
    \caption{Projections of OPF solutions onto operating cost and feasibility space.}
    \label{OPF_projections}
\end{figure}

To minimize privacy risks, this work proposes a mechanism $\tilde{{\cal M}}$ for the DSO, which returns a feasible solution $\tilde{s}$ at the expense of an optimality loss, as shown in Fig.~\ref{OPF_projections}.  A non-trivial benefit of choosing $\tilde{s}$ over $s^{\star}$ is that the former includes a particular perturbation of the optimal solution and thus carries less information about the real load data $D$. \newtext{For instance, the sub-optimal solution can feature a more intensive deployment of expensive DERs instead of purchasing less expensive power from the high-voltage grid.} To ensure that $\tilde{{\cal M}}$ returns a differentially private solution, $\tilde{s}$ has to follow a carefully calibrated noise distribution, as depicted on the right plane in Fig.~\ref{OPF_projections}. In other words, the mechanism must satisfy Definition \ref{eq:dp_def}, i.e., on adjacent load datasets, it must output distributions that differ by at most $\varepsilon$ and $\delta$ in multiplicative and additive terms, respectively. However, with smaller $\varepsilon$ and $\delta$, the variance of the OPF solutions and hence the  probability of producing an infeasible solution increases. The mechanism thus needs to address this feasibility issue. Finally, the last desired property is the ability to control the induced optimality loss $\Theta$ in order to ensure cost-effective grid operations. 

\newtext{
To provide differentially private OPF solutions, the work focuses on the randomization of active power flows as their sensitivities to grid loads can be directly upper-bounded by the load magnitudes in radial grids. Since the OPF equations \eqref{det_OPF_root}--\eqref{det_OPF_voltage} couple OPF variables, the randomization of power flows will also induce the randomization of reactive power flows and voltages. Therefore, the randomized OPF mechanism $\tilde{\mathcal{M}}$ is now seen as a mapping from a dataset $D$ to the active power flow solution $f^{p}$. Let $F^{p}\in\mathbb{R}^{|\set{L}|}$ be a particular realization of the randomized active power flows. The privacy goal of this work is to ensure that $\tilde{\mathcal{M}}$ satisfies 
$$
D \sim_\beta D': \mathbb{P}[\tilde{{\cal M}}(D) \in F^{p}] \leqslant e^\varepsilon \mathbb{P}[\tilde{{\cal M}}(D') \in F^{p}] + \delta,\;\text{i.e.,}
$$
the definition of $(\varepsilon,\delta)-$DP on $\beta-$adjacent load datasets. 
}
\section{Differentially Private OPF Mechanism}
\label{DP_OPF_mechansim}
\newtext{
This section provides a mathematical description of mechanism $\Mdp$ and details its application. Section \ref{perturbation_OPF_solution} describes the perturbation of generator outputs to attain the randomization of power flows, Section \ref{CC_OPF_solution} details the chance-constrained program that accommodates the perturbation in a feasible manner, and Section \ref{subsection:mechanism} explains the mechanism application as well as its feasibility and privacy guarantees.
}
\subsection{Random Perturbation of OPF Solutions}
\label{perturbation_OPF_solution}
\newtext{
Consider a random perturbation $\xi\in \RR^{|\set{L}|}$ which obeys a Gaussian
distribution $\mathcal{N}(0,\Sigma)$ with covariance matrix
$$\Sigma=\text{diag}([\sigma_{1}^{2},\dots,\sigma_{|\set{L}|}^{2}])=\text{diag}(\sigma^2)\in\mathbb{R}^{|\set{L}|\times|\set{L}|}.$$
Throughout the paper, $\Sigma$, $\sigma^{2}$, and $\sigma$ are used interchangeably
to discuss the perturbation parameters. The power flows are conditioned on perturbation $\xi$ when the following affine policies are imposed on DERs and substation supplies:
\begin{subequations}\label{affine_policies}
\begin{align}\label{affine_policy} \tilde{g}_{i}^{p}(\xi) = g_{i}^{p} +
\mysum{\ell\in\set{D}_{i}}\alpha_{i\ell}\xi_{\ell} -
\mysum{\ell\in\set{U}_{i}}\alpha_{i\ell}\xi_{\ell},
\;\forall i\in\set{N},
\end{align} 
where $\tilde{g}_{i}^{p}(\xi)$ and $g_{i}^{p}$ are, respectively, the random and nominal (mean) active power outputs, and
$\alpha_{i\ell}$ is the portion of random perturbation $\xi_{\ell}$ provided by the supplier at node $i$, modeled as a free variable. The policies in \eqref{affine_policy} are viable when the following balancing conditions are enforced: 
\begin{align}\label{policy_balance}
      \mysum{i\in\set{U}_{\ell}}\alpha_{i\ell}=1,
    \;\mysum{i\in\set{D}_{\ell}}\alpha_{i\ell}=1,
    \;\forall \ell\in\set{L}, 
\end{align} 
\end{subequations} 
such that for each line $\ell$, the upstream suppliers adjust their aggregated output by $\xi_{\ell}$ and the downstream DERs counterbalance this perturbation by $\xi_{\ell}$, thus satisfying power balance.}

\newtext{The policies in \eqref{affine_policies} differ from those in stochastic dispatch models in \cite{bolognani2017fast,mieth2018data,bienstock2019variance,roald2016corrective,dvorkin2020chance}, where the overall generator recourse compensates for the mismatch between grid loads and renewable forecast error realizations. However, since the affine nature of generator response remains similar, the proposed policy directly extends to balance renewable forecast errors.}

To provide a succinct representation of the randomized OPF variables, consider a topology matrix 
$T \in \mathbb{R}^{|\set{N}|\times|\set{L}|}$ whose elements are such that:
\begin{equation*}
    T_{i\ell} = \left\{
    \begin{array}{rl}
        1, & \text{if line $\ell$ is downstream w.r.t.~node $i$},\\
        -1, & \text{if line $\ell$ is upstream w.r.t.~node $i$} \\
        0, & \text{otherwise}.
    \end{array} \right.
\end{equation*}

\noindent
Consider also an auxiliary row vector 
$\rho_{i}^{p}=T_{i}\circ\alpha_{i}$ 
that returns a Schur product of $i^{\text{th}}$ row of $T$ and $i^{\text{th}}$ row of $\alpha$, and set $\rho_{i}^{q} = \rho_{i}^{p}\tan\phi_{i}$. \newtext{If the grid DERs allow, the later can be relaxed to model variable DER power factors.}
Using this notation, the perturbed OPF solution is modeled as the following set of random variables:
\begin{subequations}\label{randomized_variables}
\begin{flalign} 
\tilde{g}_{i}^{\dag}(\xi) &= g_{i}^{\dag}+\rho_{i}^{\dag}\xi, \; \forall i \in \set{N},\label{randomized_variables_gen}\\ 
\tilde{f}_{\ell}^{\dag}(\xi) &= f_{\ell}^{\dag} - \Big[\rho_{\ell}^{\dag}+\mysum{j\in\set{D}_{\ell}}\rho_{j}^{\dag}\Big]\xi, \; \forall \ell \in \set{L},\label{randomized_variables_flow}\\
\tilde{u}_{i}(\xi) &= u_{i} + 2\mysum{j\in\set{R}_{i}} \big[r_{j}\big(\rho_{j}^{p} + \mysum{k\in\set{D}_{j}}\rho_{k}^{p}) + \notag\\ 
    & \hspace{61pt}  x_{j}\big(\rho_{j}^{q} +\mysum{k\in\set{D}_{j}}\rho_{k}^{q}\big)\big]\xi, \; \forall i \in \set{L},\label{randomized_variables_voltage}
\end{flalign}
\end{subequations} 
where the randomized power flows $\tilde{f}_{i}^{\dag}$ are obtained
by substituting generator policy \eqref{affine_policy} into
\eqref{det_OPF_flow}, and randomized voltage magnitudes
$\tilde{u}_{i}$ are expressed by substituting $\tilde{f}_{i}^{\dag}$
into \eqref{det_OPF_voltage}, \newtext{refer to Appendix \ref{app:system_response} for details}.  Each variable is thus represented by
its nominal component and its random component whose realization
depends on $\xi$. \newtext{Furthermore, the random components of power flows in \eqref{randomized_variables_flow} and voltages \eqref{randomized_variables_voltage} also depend on the generator dispatch decisions $\rho^{\dag}$. Therefore, by properly calibrating the parameters of $\xi$ and finding the optimal dispatch decisions, the
randomized OPF solution in \eqref{randomized_variables} provides
the required privacy guarantees (see Section \ref{subsection:mechanism}, Theorem \ref{th:privacy})}. However, there is yet no guarantee that the randomized OPF solution in \eqref{randomized_variables} is feasible.

\subsection{The Chance-Constrained Optimization Program}
\label{CC_OPF_solution}

To obtain a feasible dispatch, the proposed mechanism uses a chance-constrained program which optimizes the affine functions in \eqref{randomized_variables} to make OPF solution feasible for any realization of random variable $\xi$ with a high probability.  The chance-constrained program is obtained by substituting variables \eqref{randomized_variables} into the base OPF model \eqref{det_OPF} and enforcing chance constraints on the grid limits. Its tractable formulation is provided in \eqref{model:cc_OPF_ref}, which is obtained considering the following reformulations.  



\subsubsection{Objective Function Reformulation}


The chance-constrained program minimizes the {\it expected} cost, which is reformulated as follows:
\[
\mathbb{E}_{\xi}\Big[\mysum{i\in\set{N}}c_{i}\tilde{g}_{i}^{p}\Big] = \mathbb{E}_{\xi}\Big[\mysum{i\in\set{N}}c_{i}(g_{i}^{\dag}+\rho_{i}^{\dag}\xi)\Big] = \mysum{i\in\set{N}} c_{i}g_{i}^{p},\;\forall i\in\set{N},
\]
due to the zero-mean distribution of $\xi$. 

\subsubsection{Inner Polygon Approximation of the Quadratic Power Flow Constraints} 
The substitution of the random power flow variables in \eqref{randomized_variables_flow} into the apparent power flow limit constraints \eqref{det_OPF_flow_limit} results in the following expression
\[
(\tilde{f}^p_\ell)^2 + (\tilde{f}^p_\ell)^2 \leq \bar{f}^2_\ell,   \;\forall\ell\in\set{L},                                                                         
\]
which exhibits a quadratic dependency on random variable $\xi$, for which no tractable chance-constrained reformulation is known. To resolve this issue, the above quadratic constraint is replaced by the inner polygon \cite{akbari2016linear,mieth2019distribution}, which writes as
\begin{align}\label{eq:inner_polygon}
    \gamma_{c}^{p}\tilde{f}_{i}^{p} + \gamma_{c}^{q}\tilde{f}_{i}^{q} + \gamma_{c}^{s}\overline{f}_{i}\leqslant0,\;\forall i\in\set{L},\forall c\in\set{C},
\end{align}
where $\gamma_{c}^{p}, \gamma_{c}^{q}, \gamma_{c}^{s}$ are the coefficients for each side $c$ of the polygon. The cardinality $|\set{C}|$ is arbitrary, but a higher cardinality brings a better accuracy. \newtext{Equation \eqref{eq:inner_polygon} is not a relaxation but an inner convex approximation: if the power flow solution is feasible for \eqref{eq:inner_polygon}, it is also feasible for the original quadratic constraint \eqref{det_OPF_flow_limit}. } 

\subsubsection{Conic Reformulation of Linear Chance Constraints} 

For the normally distributed variable $\xi$ \newtext{with known moments}, the chance constraint of the form $\mathbb{P}_{\xi}[\xi^{\top}x\leqslant b]\geqslant1-\eta$ is translated into a second-order cone constraint as \cite[Chapter 4.2.2]{boyd2004convex}:
$$ z_{\eta}\norm{\text{std}(\xi^{\top}x)}_{2}\leqslant b - \mathbb{E}_{\xi}[\xi^{\top}x],$$ 
where $z_{\eta}=\Phi^{-1}(1-\eta)$ is the inverse cumulative distribution function of the standard Gaussian distribution at the $(1-\eta)$ quantile, and $\eta$ is the constraint violation probability. 
Therefore, the individual chance constraints on the generation, voltage, and power flow variables are formulated in a conic form in \eqref{cc_OPF_ref_gen_max}--\eqref{cc_OPF_ref_flow}, respectively. The resulting tractable formulation of the chance-constrained OPF program is as follows:
\begin{subequations}\label{model:cc_OPF_ref}
\begin{align}
    &\text{CC-OPF}\colon\;\minimize{\set{V}=\{g^{\dag},f^{\dag},u,\rho^{\dag}\}}\;\mysum{i\in\set{N}}c_{i}g_{i}^{p}\label{cc_OPF_ref_obj}\\
    &\text{s.t.}\quad\text{Equations}\;\eqref{det_OPF_root}-\eqref{det_OPF_voltage},\eqref{policy_balance},\label{cc_OPF_ref_det_con}\\
    & z_{\eta^{g}}\normtwo{\big}{\rho_{i}^{\dag}\sigma}_{2}\leqslant\overline{g}_{i}^{\dag}-g_{i}^{\dag},\;\forall i\in\set{N},\label{cc_OPF_ref_gen_max}\\
    & z_{\eta^{g}}\normtwo{\big}{\rho_{i}^{\dag}\sigma}_{2}\leqslant g_{i}^{\dag} - \underline{g}_{i}^{\dag},\;\forall i\in\set{N},\label{cc_OPF_ref_gen_min}\\
    &z_{\eta^{u}}\Big\lVert\Big[\mysum{j\in\set{R}_{i}}\big[r_{j}\big(\rho_{j}^{p} +\mysum{k\in\set{D}_{j}}\rho_{k}^{p}) + x_{j}\big(\rho_{j}^{q} +\mysum{k\in\set{D}_{j}}\rho_{k}^{q}\big)\big]\Big]\sigma\Big\rVert_{2}\nonumber\\
        &\quad\leqslant \textstyle\frac{1}{2}\left(\overline{u}_{i}-u_{i}\right), \;\forall i\in\set{L},\label{cc_OPF_ref_volt_max}\\
    &z_{\eta^{u}}\Big\lVert\Big[\mysum{j\in\set{R}_{i}}\big[r_{j}\big(\rho_{j}^{p} +\mysum{k\in\set{D}_{j}}\rho_{k}^{p}) + x_{j}\big(\rho_{j}^{q} +\mysum{k\in\set{D}_{j}}\rho_{k}^{q}\big)\big]\Big]\sigma\Big\rVert_{2}\nonumber\\
        &\quad\leqslant \textstyle\frac{1}{2}\left(u_{i} - \underline{u}_{i}\right), \;\forall i\in\set{L},\label{cc_OPF_ref_volt_min}\\
    &z_{\eta^{f}}\normtwo{\Big}{\big(\gamma_{c}^{p}\big[\rho_{\ell}^{p}+\underset{{i\in\set{D}_{\ell}}}{\textstyle\sum}\rho_{i}^{p}\big]+\gamma_{c}^{q}\big[\rho_{\ell}^{q}+\underset{{i\in\set{D}_{\ell}}}{\textstyle\sum}\rho_{i}^{q}\big]\big)\sigma}_{2}\leqslant\nonumber\\
    &\quad-\gamma_{c}^{p}f_{\ell}^{p} - \gamma_{c}^{q}f_{\ell}^{q} - \gamma_{c}^{s}\overline{f}_{\ell},\;\forall \ell\in\set{L},\forall c\in\set{C}.\label{cc_OPF_ref_flow}
\end{align}
\end{subequations}

\subsection{The Privacy-Preserving Mechanism and Guarantees}
\label{subsection:mechanism}
The functioning of the privacy-preserving mechanism $\tilde{{\cal M}}$ is explained in Algorithm \ref{alg}. The Algorithm first computes the covariance matrix $\Sigma$ that encodes the DP parameters $(\varepsilon, \delta)$, and adjacency parameter $\beta$. The mechanism then solves the optimization problem in \eqref{model:cc_OPF_ref} to obtain an optimal chance-constrained solution $\optimal{\set{V}}$. Last, the mechanism samples the random perturbation and obtains the final OPF solution using equations \eqref{randomized_variables}. By design of problem \eqref{model:cc_OPF_ref}, the sampled OPF solution is guaranteed to satisfy grid limits and customer loads up to specified violation probabilities $\eta^{g},\eta^{u}$ and $\eta^{f}$, of the generator, voltage, and power flow constraints. 

\begin{algorithm}[t]
  \newtext{{\bf Input:} $D, \varepsilon, \delta, \beta, \eta^{g}, \eta^{u}, \eta^{f}$\\
  Define covariance $\Sigma=f(\varepsilon, \delta, \beta)$ as per Theorem \ref{th:privacy}\\
  Solve $\optimal{\set{V}}\;\;\leftarrow\;\text{argmin problem}\;\eqref{model:cc_OPF_ref} $ using $D$ and $\Sigma$\\
  Sample random perturbation $\hat{\xi}\sim\mathcal{N}(0,\Sigma)$\\
  Obtain final OPF solution from \eqref{randomized_variables} using $\optimal{\set{V}}$ and $\hat{\xi}$\\
  {\bf Release:} $\tilde{f}^{\dag}(\hat{\xi}),\tilde{u}(\hat{\xi}),\tilde{g}^{\dag}(\hat{\xi})$}
  \caption{DP CC-OPF mechanism $\tilde{\mathcal{M}}$}
  \label{alg}
\end{algorithm}

The privacy guarantees, in turn, depend on the specification of DP parameters $(\varepsilon,\delta)$ and the vector of adjacency coefficients $\beta.$ For simplicity, the DP parameters are assumed to be uniform for all customers and specified by the DSO, whereas customer privacy preferences are expressed in the submitted adjacency coefficients. In this setting, the load of every customer $i$ is guaranteed to be indistinguishable from any other load in the range $[d_{i}^{p}-\beta_{i},d_{i}^{p}+\beta_{i}]$ in the release of OPF solution related to node $i$ up to DP parameters $(\varepsilon,\delta)$. This guarantee is formalized by the following result. 

\begin{theorem}[Privacy Guarantees]
\label{th:privacy}
Let $(\varepsilon,\delta)\in(0,1)$ and $\sigma_{i} \geqslant
\beta_{i}\sqrt{2\text{ln}(1.25/\delta)}/\varepsilon, \;\forall
i\in\set{L}.$ Then, if problem \eqref{model:cc_OPF_ref} returns an optimal
solution, mechanism $\tilde{{\cal M}}$ is $(\varepsilon, \delta)$-differentially private for
$\beta$-adjacent load datasets. That is, the probabilities of
returning a power flow solution in set $F^{p}$ on any two
$\beta$-adjacent datasets $D$ and $D'$ are such that
\begin{equation*}
  \mathbb{P}[\tilde{{\cal M}}(D) \in F^{p}] \leqslant e^\varepsilon \mathbb{P}[\tilde{{\cal M}}(D') \in F^{p}] + \delta,
\end{equation*}
where $\mathbb{P}$ denotes the probability over runs of $\tilde{{\cal M}}$. 
\end{theorem}
\begin{proof}
The full proof is available in Appendix \ref{app:proof_privacy}
and relies on two intermediate results summarized in Lemmas
\ref{lemma_lower_bound} and \ref{lemma_flow_load_sensitivity}.  The first lemma shows that the standard deviation of power
flow related to customer $i$ is at least as much as
$\sigma_{i}$. Therefore, by specifying $\sigma_{i}$, the DSO attains
the desired degree of randomization. The second lemma shows that
$\beta_{i}\geqslant\Delta_{i}^{\beta}$, i.e., if $\sigma_{i}$ is
parameterized by $\beta_{i}$, then $\sigma_{i}$ is also parameterized
by sensitivity $\Delta_{i}^{\beta}$, required by the Gaussian
mechanism in Theorem \ref{def_gaus_mech}.
\end{proof}

\section{Mechanism Extensions}
\label{sec:properties}

\subsection{OPF Variance Control}

Due to the radial topology of distribution grids, the flow perturbations along the same radial branch induce larger flow variances than those intended by the covariance matrix $\Sigma$. This section extends the mechanism $\Mdp$ to reduce the overall flow variance while still preserving privacy guarantees. Two strategies are proposed to achieve this goal.

\subsubsection{Total Variance Minimization}
The flow standard deviation, obtained from \eqref{randomized_variables_flow}, depends on the DER participation variables $\rho^{\dag}$. Therefore, the variance of power flows can be controlled by optimizing the DER dispatch. This variance control strategy is enabled by replacing problem \eqref{model:cc_OPF_ref} at the core of mechanism $\Mdp$ by the following optimization:
\begin{subequations}\label{total_variance_control}
\begin{align}
    \text{ToV-CC-OPF}\colon\;&\minimize{\set{V}\cup \{t\}}\;
    \mysum{i\in\set{N}}c_{i}g_{i}^{p} + \mysum{\ell\in\set{L}}\psi_{\ell}t_{\ell}\\
    \text{s.t.}\quad&\normtwo{\Big}{\big[\rho_{\ell}^{p}+\mysum{i\in\set{D}_{\ell}}\rho_{i}^{p}\big]\sigma}_{2}\leqslant t_{\ell},\;\forall\ell\in\set{L},\label{total_variance_control_flow_std}\\
    &\text{Equations}\;\eqref{cc_OPF_ref_det_con}-\eqref{cc_OPF_ref_flow},
\end{align}
\end{subequations}
where the decision variable $t_{\ell}$ represents the standard deviation
of the active power flow in line $\ell$, which is penalized in the
objective function by a non-negative parameter $\psi_{\ell}$. By
choosing $\psi_{\ell},\forall \ell\in\set{L}$, the DSO minimizes the
total variance at the expense of operational cost.  Note that, by
Lemma \ref{lemma_lower_bound}, optimization
\eqref{total_variance_control} does not violate the privacy guarantees.

\subsubsection{Pursuing Target Variance}\label{sec:target_var}
This strategy solely perturbs the flow in the selected line of the radial branch (e.g., adjacent to the customer with the strongest privacy requirement) and constrains the DERs to maintain the flow variance in each line as required by the original matrix $\Sigma$. It specifies a new matrix $\hat{\Sigma}=\text{diag}([\hat{\sigma}_{1}^2,\dots,\hat{\sigma}_{|\set{L}|}^2])$, $\mathbbm{1}^{\top}\hat{\sigma}^2\leqslant \mathbbm{1}^{\top}\sigma^2$, that contains a smaller number of perturbations. This control is enabled  by replacing problem \eqref{model:cc_OPF_ref} by the following optimization:
\begin{subequations}\label{target_variance_control}
\begin{align}
    \text{TaV-CC-OPF}\colon\;&\minimize{\set{V}\cup \{t,\tau\}}\;
    \mysum{i\in\set{N}}c_{i}g_{i}^{p} + \mysum{\ell\in\set{L}}\psi_{\ell}\tau_{\ell}\\
    \text{s.t.}\quad&\normtwo{\Big}{\big[\rho_{\ell}^{p}+\mysum{i\in\set{D}_{\ell}}\rho_{i}^{p}\big]\hat{\sigma}}_{2}\leqslant t_{\ell},\;\forall\ell\in\set{L},\label{target_variance_control_flow_std}\\
    &\normtwo{\big}{t_{\ell}-\sigma_{\ell}}_{2}\leqslant \tau_{\ell},\;\forall \ell\in\set{L},\label{target_variance_control_var_distance}\\
    &\text{Equations}\;\eqref{cc_OPF_ref_det_con}-\eqref{cc_OPF_ref_flow} \;\text{with}\; \hat{\sigma},
\end{align}
\end{subequations}
where, $t_{\ell}$ returns the resulting flow standard deviation, while constraint \eqref{target_variance_control_var_distance} yields the distance $\tau_{\ell}$ between the resulting standard deviation and original one $\sigma_{\ell}=\Sigma_{\ell,\ell}^{1/2}$ required to provide customer at node $\ell$ with differential privacy. By penalizing this distance in the objective function, the DSO attains privacy at a smaller amount of random perturbations. Note, as optimization \eqref{target_variance_control} acts on covariance matrix $\hat{\Sigma}$ instead of $\Sigma$, the DSO needs to verify a posteriori that $t_{\ell}\geqslant\sigma_{\ell},\;\forall\ell\in\set{L}$. 

\subsection{Optimality Loss Control}

The application of mechanism $\Mdp$ necessarily leads to an optimality loss compared to the solution of non-private mechanism $\M$. This section slightly abuses the notation and denotes the cost of the non-private OPF solution and that of the proposed DP mechanism when evaluated on a dataset $D$ by $\M(D)$ and $\Mdp(D)$, respectively. The optimality loss $\Theta$ is measured in expectation as the $L_2$ distance, i.e., $$\mathbb{E}[\Theta]=\norm{\M(D)-\mathbb{E}[\Mdp(D)]}_{2},$$ as $\M(D)$ always provides a deterministic solution. However, the worst-case realization of $\Mdp(D)$ may significantly exceed the expected value and lead to a larger optimality loss. To this end, this section introduces the optimality loss control strategy using the Conditional Value-at-Risk (CVaR) measure \cite{shapiro2009lectures}. 


Consider $\varrho$\% of the worst-case realizations of the optimally loss. The expected value of these worst-case realizations can be modeled as a decision variable using the CVaR measure as
\begin{align}\label{CVaR_non_reformulated}
    \text{CVaR}_{\varrho} = \mu_{c} + \sigma_{c}\phi\left(\Phi^{-1}(1-\varrho)\right)/\varrho,
\end{align}
where $\mu_{c}$ and $\sigma_{c}$ represent the expected value and standard deviation of operational costs, while  $\phi$ and $\Phi^{-1}(1-\varrho)$ denote the probability density function and the inverse cumulative distribution function at the $(1-\varrho)$ quantile of the standard Normal distribution. From Section \eqref{CC_OPF_solution}, it follows that 
\begin{align*}
    \mu_{c}=\mathbb{E}[c^{\top}\tilde{g}^{p}]=\mathbb{E}[c^{\top}(g^{p}+\rho\xi)]=c^{\top}g^{p},
\end{align*}
for zero-mean $\xi$, and the standard deviation finds as
\begin{align*}
    \sigma_{c}=\text{std}[c^{\top}(g^{p}+\rho\xi)]=\text{std}[c^{\top}(\rho\xi)]=\norm{c^{\top}(\rho\sigma)}_{2},
\end{align*}
providing a convex reformulation of the CVaR in \eqref{CVaR_non_reformulated}. Therefore, for some trade-off parameter $\theta\in[0,1]$, the DSO can trade off the mean and $\text{CVaR}_{\varrho}$ of the optimality loss by substituting problem \eqref{model:cc_OPF_ref} at the core of mechanism $\Mdp$ by the following optimization CVaR-CC-OPF:
\begin{subequations}\label{cost_CVaR_optimization}
\begin{align}
    \minimize{\set{V}\cup \sigma_{c}}&(1-\theta)c^{\top}g^{p}+\theta\left[c^{\top}g^{p} + \textstyle\sigma_{c}\phi\left(\Phi^{-1}(1-\varrho)\right)/\varrho\right]\\
    \text{s.t.}\quad&\norm{c^{\top}(\rho\sigma)}_{2}\leqslant\sigma_{c},\\
    &\text{Equations}\;\eqref{cc_OPF_ref_det_con}-\eqref{cc_OPF_ref_flow},
\end{align}
\end{subequations}
where the standard deviation $\sigma_{c}$ is modeled as a decision variable. Notice, the optimality loss control by means of \eqref{cost_CVaR_optimization} does not violate the privacy
guarantees as per Lemma \ref{lemma_lower_bound}.

\section{Numerical Experiments}\label{sec:experiments}

The experiments consider a modified 15-node radial grid from \cite{papavasiliou2017analysis}, which includes network parameters taken from \cite{robert_code}, nodal loads as given in Table \ref{table_OPF_summary}, and nodal DERs with cost coefficients drawn from Normal distribution $c_{i}\sim \mathcal{N}(10,2)$ \scalebox{0.85}{\$/MWh}, generation limits $\underline{g}_{i}^{p}=0\;\scalebox{0.85}{MW}, \underline{g}_{i}^{p}=8\;\scalebox{0.85}{MW}$, and power factor $\text{tan}\phi_{i}=0.5$, $\forall i\in\set{N}$. The constraint violation probabilities are set as $\eta^{g}=1\%, \eta^{u}=2\%$ and $\eta^{f}=10\%$. The DP parameters are set to $\varepsilon\rightarrow1, \delta= 1/n = 0.071$ with $n$ being a number of grid customers, while the adjacency parameters $\beta_{i},\forall i\in\set{L},$ vary across the experiments. 
All models are implemented in Julia using the JuMP package \cite{lubin2015computing}, and all data and codes are relegated to the e-companion \cite{companion2019dvorkin}.

\subsection{Illustrative Example}\label{sec:exp_illustrative_example}
\begin{figure*}[!t]
    \begin{center}
    \resizebox{\textwidth}{!}{%
    \includegraphics{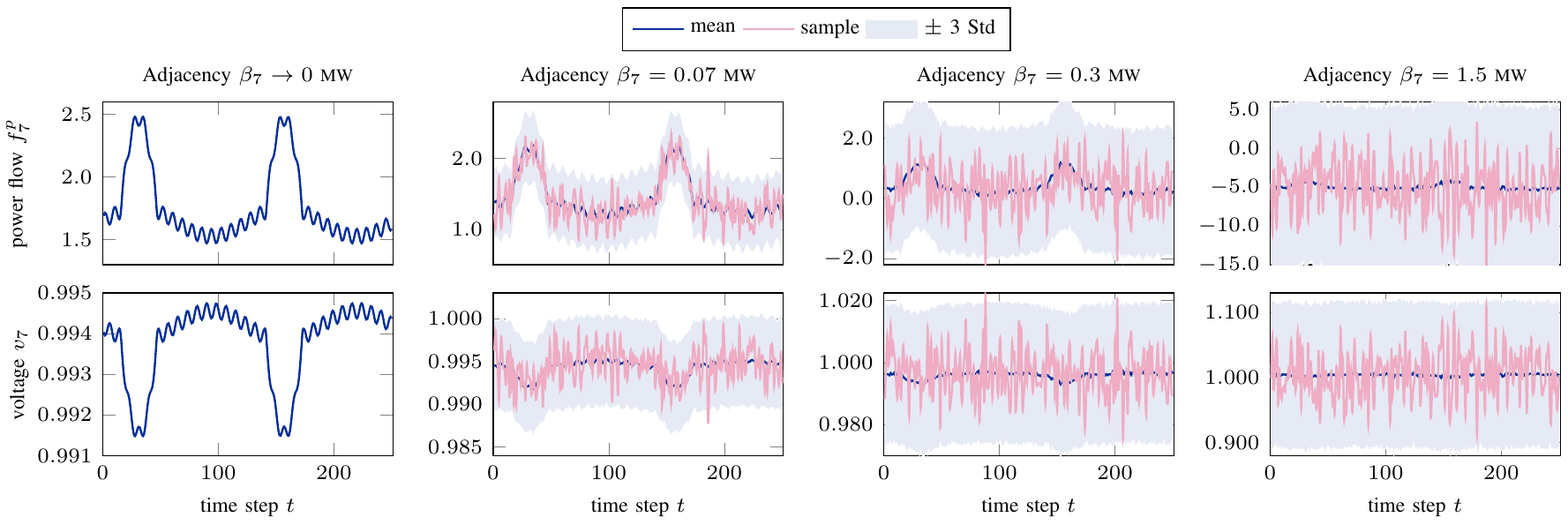}
    }
    \end{center}
    \caption{Power flow and voltage magnitude at node 7 as functions of adjacency coefficient $\beta_{7}$. The flow and voltage solutions are given by their mean value (blue) and the range of $\pm3$ standard deviations (light blue). The OPF solution implemented by the DSO is given by sample trajectories (red).}
    \label{fig_illustrative_example}
\end{figure*}

\newtext{The purpose of the illustrative example is to simulate and obfuscate periodic components of the load profile in power flow and voltage measurements.} Assume that the customer at node 7 has an atypical load pattern representing its production technology. Her pattern is obtained by multiplying the maximum load by $k(t)$, a multiplier with the following three periodic components:
\begin{align*}
    k(t)=&\textstyle\text{max}\left\{\sin{\frac{5}{10^2}t},\frac{7}{10}\right\} +\frac{5}{10^2}\sin{\frac{5}{10^2}t}+\frac{25}{10^3}\sin{\frac{75}{10^2}t}
\end{align*}
where $t$ is the time step. \newtext{The parameters of multiplier $k(t)$ are selected such that the load components have different magnitudes and frequencies.} The non-private OPF solution provided by the D-OPF model leaks the information about this pattern through the power flow $f_{7}^{p}$ and voltage $v_{7}$ readings, as displayed on the left plots in Fig. \ref{fig_illustrative_example}. To obfuscate the load pattern in the OPF solution, the customer submits the privacy preference $\beta_{7}$, which is accommodated by the DSO using mechanism $\Mdp$. Figure \ref{fig_illustrative_example} shows that by setting $\beta_{7}\rightarrow0.07$ \scalebox{0.85}{MW}, the presence of the smallest periodic component is obfuscated through randomization, while the presence of the two remaining components is still distinguished. With an increasing $\beta_{7}$, the mechanism further obfuscates the medium and largest periodic components.

\subsection{Privacy Guarantees} \label{sec:exp_privacy}
\begin{figure}
    \begin{center}
    \includegraphics{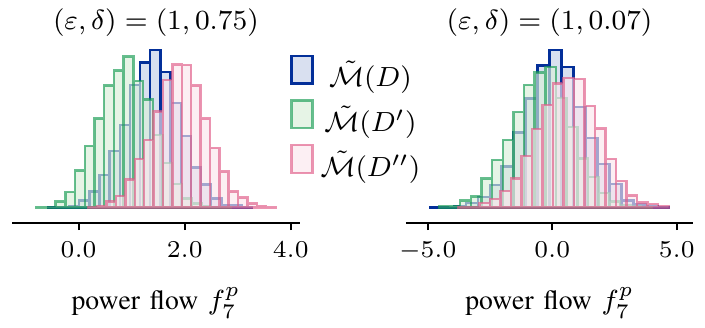}
    \end{center}
    \caption{The overlay of power flow probability densities obtained on the three $\beta_{7}-$adjacent load datasets for $\delta=0.75$ and $\delta=0.07$ (5000 samples).}
    \label{fig_privacy_feas_guarantees}
\end{figure}

To illustrate the privacy guarantees of Theorem \ref{th:privacy}, consider the same grid customer at node 7 with the load of 2.35 \scalebox{0.85}{MW}. For $\beta_{7}$, consider two adjacent
load datasets $D'$ and $D''$, containing $d_{7}^{\prime p}=d_{7}^{p}-\beta_{7}$ and $d_{7}^{\prime\prime p}=d_{7}^{p}+\beta_{7}$, respectively. The non-private OPF
mechanism returns the following power flows
\begin{align*}
    \M(D') = 2.05 \scalebox{0.85}{MW},\;
    \M(D) =& 2.35 \scalebox{0.85}{MW},\;
    \M(D'')= 2.65 \scalebox{0.85}{MW},
\end{align*} 
for $\beta_{7}=0.3$ \scalebox{0.85}{MW}, clearly distinguishing the differences in datasets through power flow readings.  The differentially private mechanism $\tilde{{\cal M}}$ in Algorithm \ref{alg}, in turn, obfuscates the load value used in the computation. Figure \ref{fig_privacy_feas_guarantees} shows that the mechanism makes the OPF solutions on the three datasets similar in the probabilistic sense, thus providing privacy guarantees for the original load dataset $D$. The maximal difference between the distributions of power flow solutions is bounded by the parameters $\varepsilon$ and $\delta$. Observe that the larger specification $\delta=0.75$ results in weaker guarantees, as  the distributions slightly stand out from one another. On the other hand, $\delta=0.07$ yields a larger noise magnitude overlapping the support of the three distributions. The OPF solution to be implemented is obtained from a single sample drawn from the blue distribution. By observing a single sample, an adversary cannot distinguish the distribution, and thus the dataset, it was sampled from. \newtext{Finally, Fig. \ref{fig_privacy_feas_guarantees} shows randomized OPF solutions obtained on a given load dataset. The parameters of the noise, however, are independent from load dataset (see Theorem \ref{th:privacy}), and the privacy guarantee for the customer at node 7 is independent from the loads and their variations at other grid nodes.}

\subsection{OPF Variance Control}\label{sec:exp_practices}

\begin{table*}[t]
\caption{Solution summary for the non-private and differentially private OPF mechanisms.}
\label{table_OPF_summary}
\begin{tabular}{@{\extracolsep{2pt}}cll|ll|llll|llll|llll@{}}
\Xhline{2\arrayrulewidth}
\\[-7.0pt]
\multicolumn{1}{c}{\multirow{3}{*}{$i$}} & \multicolumn{1}{c}{\multirow{3}{*}{$d_{i}^{p}$}} & \multicolumn{1}{c}{\multirow{3}{*}{$\sigma_{i}$}} & \multicolumn{2}{c}{D-OPF, Eq. \eqref{det_OPF}} & \multicolumn{4}{c}{CC-OPF, Eq. \eqref{model:cc_OPF_ref}} & \multicolumn{4}{c}{ToV-CC-OPF, Eq. \eqref{total_variance_control}} & \multicolumn{4}{c}{TaV-CC-OPF, Eq. \eqref{target_variance_control}} \\ 
\cline{4-5}\cline{4-5}
\cline{6-9}\cline{6-9}
\cline{10-13}\cline{10-13}
\cline{14-17}\cline{14-17}
\\[-7.0pt]
\multicolumn{1}{c}{} & \multicolumn{1}{c}{} & \multicolumn{1}{c}{} & \multicolumn{1}{c}{\multirow{2}{*}{$f_{i}^{p}$}} & \multicolumn{1}{c}{\multirow{2}{*}{$v_{i}$}} & \multicolumn{2}{c}{$f_{i}^{p}$} & \multicolumn{2}{c}{$v_{i}$} & \multicolumn{2}{c}{$f_{i}^{p}$} & \multicolumn{2}{c}{$v_{i}$} & \multicolumn{2}{c}{$f_{i}^{p}$} & \multicolumn{2}{c}{$v_{i}$} \\ 
\\[-8.0pt]
\cline{6-7}\cline{6-7}
\cline{8-9}\cline{8-9}
\cline{10-11}\cline{10-11}
\cline{12-13}\cline{12-13}
\cline{14-15}\cline{14-15}
\cline{16-17}\cline{16-17}
\multicolumn{1}{c}{} & \multicolumn{1}{c}{} & \multicolumn{1}{c}{} & \multicolumn{1}{c}{} & \multicolumn{1}{c}{} & \multicolumn{1}{c}{mean} & \multicolumn{1}{c}{std} & \multicolumn{1}{c}{mean} & \multicolumn{1}{c}{std} & \multicolumn{1}{c}{mean} & \multicolumn{1}{c}{std} & \multicolumn{1}{c}{mean} & \multicolumn{1}{c}{std} & \multicolumn{1}{c}{mean} & \multicolumn{1}{c}{std} & \multicolumn{1}{c}{mean} & \multicolumn{1}{c}{std} \\
\Xhline{2\arrayrulewidth}
0 & 0 & -- & -- & 1.00 & -- & -- & 1.00 & -- & -- & -- & 1.00 & -- & -- & -- & 1.00 & -- \\
1 & 2.01 & \textbf{0.48} & 8.5 & 1.00 & 11.3 & \textbf{2.68} & 1.00 & 0.0016 & 12.6 & \textbf{0.69} & 1.00 & 0.0004 & 13.0 & \textbf{0.48} & 1.00 & 0.0003 \\
2 & 2.01 & \textbf{0.48} & 6.5 & 1.00 & 9.3 & \textbf{2.68} & 0.99 & 0.0057 & 11.4 & \textbf{0.71} & 0.99 & 0.0015 & 11.0 & \textbf{0.48} & 0.99 & 0.0010 \\
3 & 2.01 & \textbf{0.48} & 4.4 & 1.00 & 7.3 & \textbf{2.68} & 0.99 & 0.0123 & 10.2 & \textbf{0.78} & 0.97 & 0.0033 & 9.0 & \textbf{0.48} & 0.98 & 0.0022 \\
4 & 1.73 & \textbf{0.41} & -8.0 & 1.00 & -1.4 & \textbf{1.72} & 0.99 & 0.0128 & 3.6 & \textbf{0.69} & 0.97 & 0.0034 & 1.7 & \textbf{0.41} & 0.98 & 0.0023 \\
5 & 2.91 & \textbf{0.70} & 5.1 & 1.00 & 3.1 & \textbf{0.87} & 0.99 & 0.0128 & 2.5 & \textbf{0.82} & 0.97 & 0.0035 & 1.9 & \textbf{0.70} & 0.98 & 0.0024 \\
6 & 2.19 & \textbf{0.52} & 2.2 & 1.00 & 0.1 & \textbf{0.87} & 0.99 & 0.0128 & 0.7 & \textbf{0.63} & 0.97 & 0.0038 & 1.0 & \textbf{0.52} & 0.98 & 0.0024 \\
7 & 2.35 & \textbf{0.56} & 2.3 & 0.99 & 0.9 & \textbf{0.63} & 0.98 & 0.0134 & 0.9 & \textbf{0.61} & 0.97 & 0.0039 & 1.0 & \textbf{0.56} & 0.98 & 0.0024 \\
8 & 2.35 & \textbf{0.56} & 10.5 & 0.99 & 6.7 & \textbf{1.18} & 0.98 & 0.0130 & 5.8 & \textbf{0.78} & 0.97 & 0.0036 & 6.4 & \textbf{0.56} & 0.98 & 0.0023 \\
9 & 2.29 & \textbf{0.55} & 5.8 & 0.99 & 3.5 & \textbf{0.88} & 0.98 & 0.0132 & 3.1 & \textbf{0.70} & 0.97 & 0.0037 & 3.6 & \textbf{0.55} & 0.98 & 0.0023 \\
10 & 2.17 & \textbf{0.52} & 3.5 & 0.99 & 1.2 & \textbf{0.88} & 0.98 & 0.0135 & 1.6 & \textbf{0.65} & 0.97 & 0.0038 & 1.3 & \textbf{0.52} & 0.97 & 0.0023 \\
11 & 1.32 & \textbf{0.32} & 1.3 & 0.99 & 0.4 & \textbf{0.39} & 0.98 & 0.0135 & 0.4 & \textbf{0.40} & 0.97 & 0.0038 & 0.6 & \textbf{0.32} & 0.97 & 0.0023 \\
12 & 2.01 & \textbf{0.48} & 6.5 & 1.00 & 3.6 & \textbf{1.23} & 1.00 & 0.0008 & 3.3 & \textbf{0.73} & 1.00 & 0.0004 & 3.6 & \textbf{0.48} & 1.00 & 0.0003 \\
13 & 2.24 & \textbf{0.54} & 4.5 & 0.99 & 1.6 & \textbf{1.23} & 1.00 & 0.0034 & 2.1 & \textbf{0.72} & 1.00 & 0.0019 & 3.2 & \textbf{0.54} & 0.99 & 0.0012 \\
14 & 2.24 & \textbf{0.54} & 2.2 & 0.99 & -0.6 & \textbf{1.23} & 1.00 & 0.0050 & 0.8 & \textbf{0.64} & 0.99 & 0.0027 & 1.0 & \textbf{0.54} & 0.99 & 0.0018
 \\
\Xhline{2\arrayrulewidth}
\multicolumn{3}{l|}{cost ($\mathbb{E}[\Theta]$)} & \multicolumn{2}{c|}{\$396.0 (0\%)} & \multicolumn{4}{c|}{\$428.0 (8.1\%)} & \multicolumn{4}{c|}{\$463.5 (17.1\%)} & \multicolumn{4}{c}{\$459.3 (16.0\%)}\\
\multicolumn{3}{l|}{$\sum_{i}\text{std}[f_{i}^{p}]$} & \multicolumn{2}{c|}{0 \scalebox{0.85}{MW}} & \multicolumn{4}{c|}{19.1 \scalebox{0.85}{MW}} & \multicolumn{4}{c|}{9.5 \scalebox{0.85}{MW}} & \multicolumn{4}{c}{7.1 \scalebox{0.85}{MW}}\\
\multicolumn{3}{l|}{infeas. $\hat{\eta}$} & \multicolumn{2}{c|}{0\%} & \multicolumn{4}{c|}{3.3\%} & \multicolumn{4}{c|}{6.9\%} & \multicolumn{4}{c}{5.5\%}\\
\multicolumn{3}{l|}{CPU time} & \multicolumn{2}{c|}{0.016$s$} & \multicolumn{4}{c|}{0.037$s$} & \multicolumn{4}{c|}{0.043$s$} & \multicolumn{4}{c}{0.052$s$}\\
\Xhline{2\arrayrulewidth}
\end{tabular}
\end{table*}

Consider the application of the mechanism when all grid customers have their adjacency coefficients set to 10\% of their loads. The non-private D-OPF and private OPF solutions, obtained with the variance-agnostic CC-OPF and variance-aware ToV- and TaV-CC-OPF models, are summarized in Table \ref{table_OPF_summary}: Each row $i$ presents the power flow and voltage solutions related to customer $i$, and the bottom rows report the expected operational cost, optimality loss $\mathbb{E}[\Theta]$ in \%, sum of power flow standard deviations, percentage $\hat{\eta}$ of infeasible instances on 5000 noise samples, and CPU times. 

The table shows that, unlike non-private, deterministic D-OPF model, the DP mechanisms return OPF variables as probability densities with given means and standard deviations. For all three DP mechanisms, the flow standard deviations are at least as much as those required by Theorem \ref{th:privacy}, thus providing differential privacy. However, due to the noise applied to each network flow, the flow standard deviations provided by the CC-OPF model exceed the intended quantities, e.g., by 458\% for the first customer close to the substation. To minimize the OPF variance, the ToV-CC-OPF and TaV-CC-OPF models are used with the uniform variance penalty factor $\psi_{\ell}=10^5,\forall \ell\in\set{L}.$ The ToV-CC-OPF model also perturbs each flow in the network but it alters the optimal DER dispatch to reduce the sum of flow standard deviations by 50\%. The  TaV-CC-OPF model, in turn, introduces a limited number of perturbations to lines $\{1,5-7,9,11-13\}$ and constrains the DERs to maintain the intended standard deviation $\sigma$ across the entire network, reducing the flow standard deviation by 63\%. As ToV-CC-OPF and TaV-CC-OPF prioritize the flow variance over expected cost, the models provide larger optimality loss than the CC-OPF model. 

The OPF solution to be implemented by the DSO is a sample drawn from the probability densities reported in Table \ref{table_OPF_summary}. The empirical probability of the joint constraint violation $\hat{\eta}$  demonstrates an appropriate out-of-sample performance. \newtext{However, if the DP mechanism returns an infeasible sample, the DSO may re-sample the OPF solution, yet it comes at the expense of the relaxation of privacy guarantees: every re-sampling round increases the privacy loss linearly, by $\varepsilon$, as per composition of DP \cite[Theorem 3.14]{dwork2014algorithmic}.} Finally, the privacy-preserving mechanisms keep the CPU time within the same time-frame as the standard, non-private D-OPF mechanism. 

\subsection{Optimality Loss Control}\label{sec:exp_opt_loss_control}


\begin{figure}
\centering
\resizebox{0.45\textwidth}{!}{%
\includegraphics{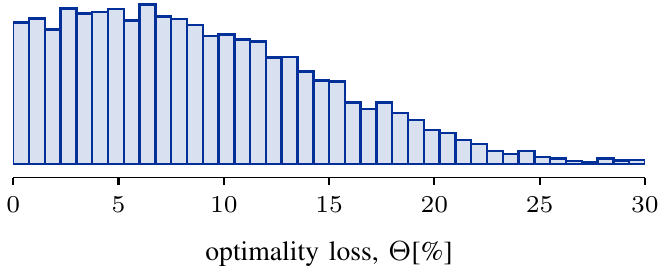}
}
\newtext{\caption{CVaR-CC-OPF: Empirical out-of-sample density of the optimality loss for a trade-off parameter $\theta=0$ (5000 samples)}}
\label{opt_loss_dist}
\end{figure}

\begin{table}
\centering
\caption{Trade-offs of the expected and $\text{CVaR}_{10\%}$ performance}
\label{CVaR_results}
\begin{tabular}{@{\extracolsep{2pt}}cccccc@{}}
\Xhline{2\arrayrulewidth}
\multicolumn{1}{c}{\multirow{2}{*}{$\theta$}} & \multicolumn{2}{c}{exp. value} & \multicolumn{2}{c}{$\text{CVaR}_{10\%}$} & \multicolumn{1}{c}{\multirow{2}{*}{$\sum_{i}\text{std}[f_{i}^{p}]$, \scalebox{0.85}{MW}}} \\
\cline{2-3}\cline{4-5}
\multicolumn{1}{c}{} & \multicolumn{1}{c}{cost, $\$$} & \multicolumn{1}{c}{$\Theta$,\%} & \multicolumn{1}{c}{cost, $\$$} & \multicolumn{1}{c}{$\Theta$,\%} & \multicolumn{1}{c}{} \\
\Xhline{2\arrayrulewidth}
0.0 & 428.0 & 8.1 & 478.1 & 20.7 & 19.1  \\
0.1 & 428.0 & 8.1 & 476.3 & 20.3 & 19.4  \\
0.2 & 428.3 & 8.2 & 475.0 & 19.9 & 19.6  \\
0.3 & 428.9 & 8.3 & 473.3 & 19.5 & 19.8  \\
0.4 & 431.9 & 9.1 & 467.8 & 18.1 & 17.3  \\
0.5 & 434.5 & 9.7 & 464.4 & 17.3 & 15.7  \\
0.6 & 438.2 & 10.7 & 461.7 & 16.6 & 14.6 \\
0.7 & 452.9 & 14.4 & 452.9 & 14.4 & 13.0 \\
\Xhline{2\arrayrulewidth}
\end{tabular}
\end{table}

\newtext{
The application of mechanism $\Mdp$ yields an optimality loss with respect to the solution of the standard, non-private OPF mechanism. For the same setting as in the previous experiment, the out-of-sample empirical distribution of the cost is depicted in Fig. \ref{opt_loss_dist}. Observe that the probability mass is centered at 8.1\% of optimality loss and that the distribution is biased towards smaller optimality losses. Although the probability of the worst-case outcomes is shown very small on out-of-sample, the expected value over 10\% of the worst-case outcomes is at relatively large 20.7\%. The DSO, however, is capable to trade off between the expected and the worst-case optimality loss by substituting the CC-OPF model in mechanism $\Mdp$ by the CVaR-CC-OPF model in \eqref{cost_CVaR_optimization}. 
} For a trade-off parameter $\theta\in[0,1]$, the expected optimality loss in $\varrho=10\%$ of the worst-case scenarios is contrasted with the expected loss in Table \ref{CVaR_results}. For $\theta=0$, the $\text{CVaR}_{10\%}$ significantly exceeds the expected value. However, by increasing $\theta$, the DSO alters the DER dispatch to reduce the worst-case optimality loss at the expense of increasing the expected value. For $\theta\geqslant0.7$, the expected value corresponds to $\text{CVaR}_{10\%}$, thus providing differential privacy at a fixed cost. Eventually, the choice of $\theta$ is driven by the DSO's risk preference. 
\newtext{Finally, the optimality loss can be further reduced by relaxing the feasibility guarantee with larger probabilities $\eta$, though it may result in an increasing out-of-sample violation probability $\hat{\eta}$. 
}

\newtext{
\subsection{Comparison with the Output Perturbation Mechanism}
It remains to compare the proposed privacy-preserving OPF mechanism with the standard, non-adapted to the specifics of OPF problems, output perturbation (OP) mechanism \cite{dwork2014algorithmic}. The functioning of this mechanism is depicted in Algorithm 2: it solves the deterministic OPF problem, perturbs the optimal power flow solution with the random noise, and then finds the feasible DER and substation dispatch that satisfies AC-OPF equations for the perturbed values of power flows. If exists, the mechanism returns an $(\varepsilon,\delta)-$differentially private OPF solution on $\beta-$adjacent datasets, or reports infeasibility otherwise. Observe, unlike the proposed mechanism in Algorithm \ref{alg}, the output perturbation mechanism does not offer feasibility guarantees, because the perturbation step is done independently from the OPF computations. 
}

\newtext{
To compare the two mechanisms, consider the provision of differential privacy for sets of nodes $1:n$, i.e., from 1 to $n$, for which $\beta_{i}\rightarrow 10\%$. By increasing $n$, the amount of noise that the DSO needs to accommodate in the grid increases. Table III summarizes the feasibility statistics for the two algorithms. Observe, even for a single customer, the output perturbation mechanism produces infeasible solutions in 52.1\% of instances, and its performance further reduces in $n$. By optimizing affine functions in \eqref{randomized_variables}, the proposed mechanism, instead, produces feasible OPF instances with a high probability, e.g. 96.7\% for the entire set of customers.}

\begin{algorithm}[t]
  \newtext{{\bf Input:} $D, \varepsilon, \delta, \beta$\\
  Solve $\{\optimal{f_{\ell}^{p}}\}_{\forall \ell\in\set{L}}\;\;\leftarrow\;\text{argmin problem}\;\eqref{det_OPF}$ using $D$ \\
  Sample $\hat{\xi}_{\ell}\sim\mathcal{N}(0,\beta_{\ell}\sqrt{2\text{ln}(1.25/\delta)}/\varepsilon), \forall \ell \in\set{L}$ \\
  Perturb power flows $\hat{f}_{\ell}^{p} = \optimal{f_{\ell}^{p}} + \hat{\xi}_{\ell}, \forall \ell \in\set{L}$\\
  Re-solve problem \eqref{det_OPF} for the fixed $f_{\ell}^{p}=\hat{f}_{\ell}^{p}, \forall \ell \in\set{L}$\\
  \eIf{\normalfont problem \eqref{det_OPF} feasible}{Implement perturbed solution}{Report infeasibility}}
  \caption{Output Perturbation (OP) Mechanism}
  \label{alg_OP_2134}
\end{algorithm}

\begin{table}[]
\centering
\newtext{\caption{Percentage of infeasible OPF instances (5000 samples) [\%]}}
\label{OP_comparison}
\newtext{
\begin{tabular}{llllllll}
\Xhline{2\arrayrulewidth}
\multirow{2}{*}{Mechanism} & \multicolumn{7}{c}{Node set $1:n$ (with non-zero adjacency $\beta_{i}$)} \\
\cline{2-8}
 & 1 & 1:2 & 1:3 & 1:4 & 1:5 & ... & 1:14 \\
\Xhline{2\arrayrulewidth}
OP & 52.1 & 87.0 & 97.9 & 99.7 & 100 & ... & 100 \\
CC-OPF & 0.1 & 0.9 & 0.9 & 1.0 & 1.1 & ... & 3.3 \\
\Xhline{2\arrayrulewidth}
\end{tabular}
}
\end{table}

\section{Conclusions and Future Work}\label{sec_conc}

This paper introduced a differentially private OPF mechanism for distribution grids, which provides formal privacy guarantees for grid customer loads. The mechanism parametrizes OPF variables as affine functions of a carefully calibrated noise to weaken the correlations between grid loads and OPF variables, thus preventing the recovery of customer loads from the voltage and power flow measurements. Furthermore, the mechanism was extended to enable the DSO to control the OPF variance induced by the noise in the computations, providing better practices for systems with more emphasis on component overloads than on operational costs. Finally, the optimality loss induced by the mechanism translates into privacy costs. To minimize the risk of large privacy costs, the mechanism was extended to enable the trade-off between the expected and worst-case performances. 

\newtext{There are several avenues for future work. To understand the impacts of the privacy preservation on distribution electricity pricing, one can explore the connection between DP parameters and distribution locational marginal prices following price decomposition approach from \cite{mieth2019distribution} and \cite{papavasiliou2017analysis}. Alternatively, the coalition game theory can be used to find an adequate privacy cost allocation among customers, similar to game-theoretic reserve cost allocation in \cite{karaca2019enabling}. Finally, the private OPF mechanism has been developed in a technically suitable ecosystem: it builds upon LinDistFlow OPF equations neglecting the effects of power losses, adopts a constant DER power factor, and does not include the control of stochastic DERs. Although these three factors are well-studied in the context of the chance-constrained OPF problems under uncertainty, it remains valid, if not crucial, for future work to explore their effects on the limits of the differential privacy provision in distribution grids.}

\appendix
\newtext{
\subsection{System Response to the Random Perturbation}\label{app:system_response}
The affine dependency of power flows on the random perturbation is obtain by substituting generator response policy \eqref{affine_policy} into OPF equations \eqref{det_OPF_flow}, that is:
\begin{subequations}\label{app_responses}
\begin{align}
    \tilde{f}_{\ell}^{\dag}(\xi)=& d_{\ell}^{\dag}-g_{\ell}^{\dag}(\xi) + \mysum{i\in\set{D}_{\ell}}(d_{i}^{\dag}-g_{i}^{\dag}(\xi))\nonumber\\
    =& d_{\ell}^{\dag} - g_{\ell}^{\dag} - \rho_{\ell}^{\dag}\xi + \mysum{i\in\set{D}_{\ell}}(d_{i}^{\dag}- g_{i}^{\dag} - \rho_{i}^{\dag}\xi)\nonumber\\
    =& d_{\ell}^{\dag} - g_{\ell}^{\dag} + \mysum{i\in\set{D}_{\ell}}(d_{i}^{\dag}- g_{i}^{\dag}) - [\rho_{\ell}^{\dag}\xi + \mysum{i\in\set{D}_{\ell}}\rho_{i}^{\dag}\xi]  \nonumber\\
    \overset{\text{due to (1c)}}{=} & f_{\ell}^{\dag} - \Big[\rho_{\ell}^{\dag}+\mysum{j\in\set{D}_{\ell}}\rho_{j}^{\dag}\Big]\xi, \; \forall \ell \in \set{L} \label{app_response_flow},
\end{align}
where $f_{\ell}^{\dag}$ is the nominal (average) component and the last term is the random flow component. The affine dependency of voltages on the random perturbation is obtain by substituting \eqref{app_response_flow} into voltage drop equation \eqref{det_OPF_voltage}, that is:
\begin{align}
    \tilde{u}_{i}(\xi) =& u_{0} - 2\mysum{\ell \in\set{R}_{i}}(\tilde{f}_{\ell}^{p}(\xi)r_{\ell}+\tilde{f}_{\ell}^{q}(\xi)x_{\ell})\nonumber\\
    =&  u_{0} - 2\mysum{\ell \in\set{R}_{i}}(f_{\ell}^{p}r_{\ell}+f_{\ell}^{q}x_{\ell})  \nonumber\\
    &+2\mysum{\ell \in\set{R}_{i}}\big[r_{\ell}\big(\rho_{\ell}^{p} + \mysum{k\in\set{D}_{\ell}}\rho_{k}^{p}\big) + x_{\ell}\big(\rho_{\ell}^{q} + \mysum{k\in\set{D}_{\ell}}\rho_{k}^{q}\big) \big]\xi\nonumber\\
    \overset{\text{due to (1d)}}{=} & u_{i} + 2\mysum{\ell \in\set{R}_{i}}\big[r_{\ell}\big(\rho_{\ell}^{p} + \mysum{k\in\set{D}_{\ell}}\rho_{k}^{p}\big) + x_{\ell}\big(\rho_{\ell}^{q} + \mysum{k\in\set{D}_{\ell}}\rho_{k}^{q}\big) \big]\xi,
\end{align} 
where $u_{i}$ is the nominal (average) component and the last term is the random voltage component. Together with the response policy in \eqref{affine_policy}, equations \eqref{app_responses} constitute the model of system response to the random perturbation, given in equations \eqref{randomized_variables}. 
\end{subequations}
}

\subsection{Proof of Theorem \ref{th:privacy}}\label{app:proof_privacy}
The proof of Theorem \ref{th:privacy} relies on Lemmas \ref{lemma_lower_bound} and \ref{lemma_flow_load_sensitivity}. The first lemma shows that the standard deviation of power flow related to customer $i$ is at least as much as $\sigma_{i}$. Therefore, by specifying $\sigma_{i}$, the DSO attains the desired degree of randomization.
\begin{lemma}\label{lemma_lower_bound} 
    If OPF mechanism \eqref{model:cc_OPF_ref} returns optimal solution, then $\sigma_{\ell}$ is the lower bound on $\text{std}[\tilde{f}_{\ell}^{p}]$.
\end{lemma}
\begin{proof}
    Consider a single flow perturbation with $\xi_{\ell}\sim\mathcal{N}(0,\sigma_{\ell}^2)$ and $\xi_{j}=0,\;\forall j\in\set{L}\backslash \ell$. The standards deviation of active power flow \eqref{randomized_variables_flow} in optimum finds as
    \begin{align}
        &\text{std}\Big[\opt{f_{\ell}^{p}} - \Big[\overset{\star}{\rho_{\ell}^{p}}+\mysum{j\in\set{D}_{\ell}}\overset{\star}{\rho_{j}^{p}}\Big]\xi\Big]=\text{std}\Big[\Big[\overset{\star}{\rho_{\ell}^{p}}+\mysum{j\in\set{D}_{\ell}}\overset{\star}{\rho_{j}^{p}}\Big]\xi\Big]\nonumber\\
        &=\text{std}\Big[\mysum{j\in\set{D}_{\ell}}\overset{\star}{\alpha}_{j\ell}\xi_{\ell}\Big]\overset{\eqref{policy_balance}}{=}
        \text{std}\Big[\xi_{\ell}\Big]=
        \sigma_{\ell},
    \end{align}
    where the second to the last equality follows from balancing conditions \eqref{policy_balance}.
    As for any pair $(\ell,j)\in\set{L}$ the covariance matrix returns $\Sigma_{\ell,j}=0$, $\sigma_{\ell}$ is a lower bound on $\text{std}[\tilde{f}_{\ell}^{p}]$ in the optimum for any additional perturbation in the network.
\end{proof}
\begin{remark}
The result of Lemma \ref{lemma_lower_bound} holds independently from the choice of objective function and is solely driven by the feasibility conditions. 
\end{remark}
The second lemma shows that $\beta_{i}\geqslant\Delta_{i}^{\beta}$, i.e., if $\sigma_{i}$ is parameterized by $\beta_{i}$, then $\sigma_{i}$ is also parameterized by sensitivity $\Delta_{i}^{\beta}$.
\begin{lemma}
\label{lemma_flow_load_sensitivity} 
	Let $D$ and $D'$ be two adjacent datasets differing in at most one load $d_i^p$ by at most $\beta_i > 0$. Then, 
	$$
	\Delta_i^\beta = \max_{\ell \in \set{L}} \norm{\M(D)_{|f_\ell^p} - \M(D')_{|f_\ell^p}}_{2} \leqslant \beta_i,
	$$
	where the notation $\M(\cdot)_{|f_\ell^p}$ denotes the value of the optimal active power flow on line $\ell$ returned by the computation $\M(\cdot)$.
\end{lemma}
\begin{proof} 
    Let $\opt{f_\ell^p}$ be the optimal solution for the active power flow in line $\ell$ obtained on input dataset $D = (d_1^p, \ldots, d_n^p)$. 
    From OPF equation \eqref{det_OPF_flow}, it can be written as
    $$
    	\opt{f_\ell^p}=d_{\ell}^{p} - \opt{g_{\ell}^{p}} + \mysum{i\in\set{D}_{\ell}}(d_{i}^{p}-\opt{g_{i}^{p}}),
    $$
    which expresses the flow as a function of the downstream loads and the optimal DER dispatch. 
    A change in the active load $d_\ell^p$ translates into a change of power flow as 
    \begin{align}
    \frac{\partial \opt{f_\ell^p}}{\partial d_{\ell}^{p}} &= 
    \underbrace{
        \frac{\partial d_{\ell}^{p}} {\partial d_{\ell}^{p}}}_{1} 
        - 
        \frac{\partial\opt{g_\ell^p}}{\partial d_{\ell}^{p}} 
        + 
        \mysum{i\in\set{D}_{\ell}} 
        \Big(
        \underbrace{
        \frac{\partial d_{i}^{p}} {\partial d_{\ell}^{p}}}_{0}
        -
        \frac{\partial\opt{g_i^p}} {\partial d_{\ell}^{p}}
        \Big)\nonumber\\
        &= 
        1 - \frac{\partial\opt{g_\ell^p}}{\partial d_{\ell}^{p}}
          - \mysum{i\in\set{D}_{\ell}}
            \frac{\partial\opt{g_i^p}} {\partial d_{\ell}^{p}},\label{eq_ap_2_summ}
    \end{align}
    where the last two terms are always non-negative due to convexity of model \eqref{det_OPF}.  The value of \eqref{eq_ap_2_summ} attains maximum when
    \begin{equation}
    \label{eq:app_2}
        \opt{g_{k}^{p}} = \overline{g}_{k}^{p}
        \mapsto 
        \frac{\partial \opt{g_{k}^{p}}} {\partial d_{\ell}^{p}} 
        = 0,
        \;\; \forall k \in \{\ell\} \cup \set{D}_{\ell}.
    \end{equation}

    Therefore, by combining \eqref{eq_ap_2_summ} with \eqref{eq:app_2} we obtain the maximal change of power flows as
    $$
    \frac{\partial \opt{f_{\ell}^{p}}} {\partial d_{\ell}^{p}}= 1.
    $$
    Since the dataset adjacency relation considers loads $d_\ell^p$ that differ by at most $\beta_\ell$, it suffices to multiply the above by $\beta_\ell$ to attain the result. It finds similarly that for a $\beta_i$ change of any load $i\in\set{N}$, all network flows change by at most $\beta_i$.
\end{proof}

\begin{proof}[Proof of Theorem \ref{th:privacy}]
Consider a customer at non-root node $i$. Mechanism $\tilde{{\cal M}}$
induces a perturbation on the active power flow $f_{i}^{p}$ by a random variable $\xi_{i}\sim\mathcal{N}(0,\sigma_{i}^{2})$. The randomized active power flow $f_{i}^{p}$ is then given as follows:
\begin{align*} 
    \tilde{f}_{i}^{p} &= 
    \opt{f_{i}^{p}} - \big[\opt{\rho_{i}^{p}} 
    + \mysum{j\in\set{D}_{i}} \opt{\rho_{j}^{p}}\big]\xi,
\end{align*} 
where $\star$ denotes optimal solution for optimization variables. For privacy parameters $(\varepsilon,\delta)$, the mechanism specifies 
$$
    \sigma_{i} \geqslant \beta_{i}\sqrt{2\text{ln}(1.25/\delta)}/\varepsilon, \;\forall i\in\set{L}.
$$
As per Lemma \ref{lemma_lower_bound}, we know that $\sigma_{i}$ is the lower bound on the standard deviation of power flow $f_{\ell}^{p}$. From Lemma \ref{lemma_flow_load_sensitivity} we also know that the sensitivity $\Delta_{i}^{\beta}$ of power flow in line $i$ to load $d_{i}^{p}$  is upper-bounded by $\beta_{i}$, so we have
\begin{align*}
    \label{power_flow_min_std}
    \text{std}[\tilde{f}_{i}^{p}] &\geqslant\sigma_{i}\geqslant\Delta_{i}^{\beta}\sqrt{2\text{ln}(1.25/\delta)}/\varepsilon.
\end{align*} 
Since the randomized power flow follow is now given by a Normal distribution with the standard deviation $ \text{std}[\tilde{f}_{i}^{p}]$ as above, by Theorem \ref{def_gaus_mech}, mechanism $\tilde{{\cal M}}$ satisfies $(\varepsilon,\delta)$-differential privacy for each grid customer up to adjacency parameter $\beta$.
\end{proof}

\balance
\bibliographystyle{IEEEtran}
\bibliography{references}

\endgroup
\end{document}


\title{Differentially Private Optimal Power Flow \\ for Distribution Grids}

\maketitle
\IEEEpeerreviewmaketitle
\begingroup
\allowdisplaybreaks

\appendix

The first result shows that the standard deviation of each perturbed power flow $\tilde{f}_\ell$, returned by an application of the chance-constrained OPF operator $\OPF^{CC}$ is lower bounded by $\sigma_\ell$.

\begin{lemma}\label{lemma_lower_bound} 
    If chance-constrained OPF Model \ref{model:cc_opf_ref} returns optimal solution, then $\sigma_{\ell}$ is the lower bound on $\text{Std}[\tilde{f}_{\ell}^{p}]$.
\end{lemma}
\begin{proof}
    Consider a single flow perturbation with $\xi_{\ell}\sim\mathcal{N}(0,\sigma_{\ell}^2)$ and $\xi_{j}=0,\;\forall j\in\set{L}\backslash \ell$. The standards deviation of active power flow \eqref{randomized_variables_flow} in optimum finds as
    \begin{align}
        &\text{Std}\Big[\fopt_{\ell}^{p} - \Big[\overset{\star}{\rho}_{\ell}^{p}+\mysum{j\in\set{D}_{\ell}}\overset{\star}{\rho}_{j}^{p}\Big]\xi\Big]=\text{Std}\Big[\Big[\overset{\star}{\rho}_{\ell}^{p}+\mysum{j\in\set{D}_{\ell}}\overset{\star}{\rho}_{j}^{p}\Big]\xi\Big]=\nonumber\\
        &\text{Std}\Big[\mysum{j\in\set{D}_{\ell}}\overset{\star}{\alpha}_{j\ell}\xi_{\ell}\Big]\overset{\eqref{policy_balance}}{=}
        \text{Std}\Big[\xi_{\ell}\Big]=
        \sigma_{\ell}\label{std_lower_bound}
    \end{align}
    where the second to the last equality follows from balancing condition \eqref{policy_balance}.
    As for any pair $(\ell,j)\in\set{L}$ covariance matrix returns $\Sigma_{\ell,j}=0$, \eqref{std_lower_bound} is a lower bound on $\text{Std}[\tilde{f}_{\ell}^{p}]$ in the optimum for any additional perturbation in the network. 
\end{proof}

The next results shows that the sensitivity of the optimal power flow to a network load $d_{\ell}^{i}$ is upper-bounded.
\begin{lemma}
\label{lemma_flow_load_sensitivity} 
	Let $D$ and $D'$ be two adjacent datasets differing in at most one load $d_i^p$ by at most quantity $\beta_i \geq 0$. Then, 
	$$
	\Delta_i^\beta = \max_{\ell \in \set{L}} | \OPF(D)_{|f_\ell^p} - \OPF(D')_{|f_\ell^p}| \leq \beta_i
	$$
	where the notation $\OPF(\cdot)_{|f_\ell^p}$ denotes the value of the optimal active power flow on line $\ell$ returned by the computation ($\OPF(\cdot)$).
\end{lemma}

\begin{proof}
    Let $\fopt_\ell^p$ denote the optimal active power flow on line $\ell$ on input $D = (d_1^p, \ldots, d_n^p)$. 
    By Equation \eqref{det_OPF_flow}, $\fopt_\ell^p$ can be written as
    $$
    	\fopt_{\ell}^{p}=d_{\ell}^{p} - \gopt_{\ell}^{p} + \mysum{i\in\set{D}_{\ell}}(d_{i}^{p}-\gopt_{i}^{p}),
    $$
    which expresses the flow $f_\ell$ as a function of the downstream loads and the optimal DER dispatch. 
    A unit change $\partial d_{\ell}^{p}$ in the active load $d_\ell^p$ translates into a change of power flow described as 
    \begin{subequations}
    \label{eq:ap_1}
    \begin{align}
    \frac{\partial \fopt_{\ell}^{p}} {\partial d_{\ell}^{p}} &= 
    \underbrace{
        \frac{\partial d_{\ell}^{p}} {\partial d_{\ell}^{p}}}_{1} 
        - 
        \frac{\partial\gopt_{\ell}^{p}}{\partial d_{\ell}^{p}} 
        + 
        \mysum{i\in\set{D}_{\ell}} 
        \Big(
        \underbrace{
        \frac{\partial d_{i}^{p}} {\partial d_{\ell}^{p}}}_{0}
        -
        \frac{\partial\gopt_{i}^{p}} {\partial d_{\ell}^{p}}
        \Big) \\ 
        &= 
        1 - \frac{\partial\gopt_{\ell}^{p}}{\partial d_{\ell}^{p}}
          - \mysum{i\in\set{D}_{\ell}}
            \frac{\partial\gopt_{i}^{p}} {\partial d_{\ell}^{p}},
    \end{align}
    \end{subequations}
    where the last two terms are always non-negative due to convexity of Model \ref{model:det_opf}. 
    The first term in the summation is 0 as no loads, except for $d_\ell$ changes.  

    In the worst case, we have that
    \begin{equation}
    \label{eq:app_2}
        \gopt_{k}^{p} = \overline{g}_{k}^{p}
        \mapsto 
        \frac{\partial \gopt_{k}^{p}} {\partial d_{\ell}^{p}} 
        = 0,
        \;\; \forall k \in \{\ell\} \cup \set{D}_{\ell},
    \end{equation}

    Combining \eqref{eq:ap_1} with \eqref{eq:app_2} results in:
    $$
    \frac{\partial \fopt_{\ell}^{p}} {\partial d_{\ell}^{p}}= 1.
    $$
    Since the dataset adjacency relation considers loads $d_\ell^p$ that differ by at most $\beta_\ell$, it suffices to multiply the above by $\beta_\ell$ to attain the result. 

    The above implies that for a $\beta_i$ change on a load on node $i$, all upstream optimal power flows will change by at most $\beta_i$. 
    In turn, the maximal change of optimal power flow for a given node $i$ is $\beta_i$.
\end{proof}

Since the privacy model adopted in this paper uses a metric space with $\beta= (\beta_1, \ldots, \beta_n)$ to define the adjacency relation between two datasets, the sensitivity of the power flow $\fopt_\ell^p$ to load $d_{\ell}^p$ is $\beta_\ell$.

It remain to show that mechanism $\OPF^CC$ achieves differential privacy. 

\begin{theorem}
Let $\epsilon > 0$ and $\delta \in [0,1]$ and consider a vector $\beta = (\beta_1, \ldots, \beta_n)$ with $\beta_i >0$ ($i \in [n]$).
Mechanism $\OPF^{CC}$ is $(\epsilon, \delta)$-differentially private for $\beta$-adjacent load datasets.
\end{theorem}

\begin{proof}

Consider a customer at non-root node $i$. Mechanism $\OPF^{CC}$
induces a perturbation on the active power flow $f_{i}^{p}$ by a random variable $\xi_{i}\sim\mathcal{N}(0,\sigma_{i}^{2})$, and dispatches available DERs according to OPF model \eqref{cc_OPF_reformulated}. 
The randomized active power flow $f_i$ is then given as follows:
\begin{align*} 
    \tilde{f}_{i}^{p} &= 
    \opt{f_{i}^{p}} - \big[\opt{\rho_{i}^{p}} 
    + \mysum{j\in\set{D}_{i}} \opt{\rho_{j}^{p}}\big]\xi
%
\end{align*} 
where quantities $\opt{(\cdot)}$ fix the optimization variables to the optimal solution of the chance-constrained program $\tilde{\OPF}^{CC}$ defined in \eqref{model:cc_OPF_ref}. 

Notice that, for fixed $\epsilon$ and $\delta$ values, for each $i \in \set{N}$, mechanism $\OPF^{CC}$ ensures that
$$
    \sigma_{i} \geq \frac{\beta_i \sqrt{2\text{ln}(1.25/\delta)}}{\epsilon}
$$
as illustrated in \eqref{cc_OPF_reformulated}.
It follows that:
\begin{align}
    \label{power_flow_min_std}
    \text{Std}[\tilde{f}_{i}^{p}] &= 
        \text{Std}\Big[\big[\opt{\rho_{i}^{p}}
            + \mysum{j\in\set{D}_{i}}
                \opt{\rho_{j}^{p}}\big]\xi\Big]\\
        &> \frac{\Delta_{i}^{\beta}\sqrt{2\text{ln}(1.25/\delta)}}{\epsilon}
        \geq \sigma_i.
\end{align} 

The above ensures that the standard deviation of the active power flow $\tilde{f}_{i}^{p}$ is lower bounded by $\sigma_i$. Since the power flow follow a Normal distribution centered in $\opt{f^p_i}$ and with standard deviation as above, by Theorem \ref{def_gaus_mech}, mechanism $\OPF^{CC}$ satisfies $(\epsilon,\delta)$-differential privacy. 
\end{proof}

\endgroup